\documentclass[a4paper,11pt,leqno]{amsart}
\usepackage[big]{haiman_style}
\usepackage{comment,longtable}

\title{Haiman's Conjecture and Springer's Representations}

\author[Minh-T\^{a}m Quang Trinh]{Minh-T\^{a}m Quang Trinh}
\address{Department of Mathematics, Howard University, Washington, DC 20059}
\email{minhtam.trinh@howard.edu}

\begin{document}
	
\maketitle

\begin{abstract}\frenchspacing{For any connected complex reductive group $G$ and element $z$ of its Weyl group $W$, we use work of Lusztig and Abreu--Nigro to compute the graded $W$-character of the intersection cohomology of any closed Lusztig variety for $z$ over the regular semisimple locus of $G$.
We relate the resulting formula to unipotent Lusztig varieties, giving a new geometric model for unicellular LLT polynomials.
We then consider Laurent polynomials $\alpha_{\psi, G}^z$ indexed by irreducible characters $\psi$, encoding how our formula decomposes into ungraded characters arising from the Springer theory of $G$.
From evidence in low rank, we conjecture that if $\psi$ is inflated from type $A$ in a particular way, then the nonzero coefficients of $\alpha_{\psi, G}^z$ are positive and unimodal.
This offers an answer to a 1993 question of Haiman about generalizing a conjecture he posed for symmetric groups.
We also prove that the matrix formed by the $\alpha_{\psi, G}^z$ is partially triangular, and that their positivity and unimodality properties are stable under inclusions of Levi subgroups.}\end{abstract}

%\setcounter{tocdepth}{1}
%\tableofcontents
\thispagestyle{empty}

%------------------------------------------------------------

%	\mainmatter

\section{Conjectures and Results}\label[section]{sec:intro}

\subsection{Motivation}

The subject of this paper is how to generalize a 1993 conjecture of Haiman from symmetric groups to Weyl groups.

For any finite Coxeter group $W$, let $H_W$ be the Hecke algebra of $W$ over $\bb{Z}[\sf{v}^{\pm 1}]$ with standard basis $\{\delta_w\}_{w \in W}$.
We fix a system of simple reflections $S \subseteq W$ and normalize the elements $\delta_w$ so that $H_W$, as an algebra, is generated by the set $\{\delta_s\}_{s \in S}$ modulo the relations $\delta_s^2 = (\sf{v}^2 - 1)\delta_s + \sf{v}^2$ and so-called braid relations.
Let $\{c_w\}_{w \in W}$ be the Kazhdan--Lusztig basis of $H_W$ denoted $\{C'_w\}_{w \in W}$ in~\cite{kl}.
It is invariant under the ring anti-automorphism sending $\sf{v} \mapsto \sf{v}^{-1}$ and $\delta_s \mapsto \delta_s^{-1}$ for all $s \in S$.
For instance, if $s \in S$, then $c_s = \sf{v}^{-1}(1 + \delta_s) = \sf{v}(1 + \delta_s^{-1})$.

Let $\Irr(W)$ be the set of irreducible characters of $W$.
The representation theory of $H_W$ deforms that of $W$, just as $H_W$ itself deforms the group ring $\bb{Z}W$.
In particular, each class function $\chi$ on $W$ deforms to a $\bb{Z}[\sf{v}^{\pm 1}]$-linear function $\chi_\sf{v}$ on $H_W$ that remains a trace: \emph{i.e.}, vanishes on commutators.
If $W$ is a Weyl group and $\chi \in \Irr(W)$, then $\chi$ takes values in $\bb{Z}$, and $\chi_\sf{v}$ takes values in $\bb{Z}[\sf{v}^{\pm 1}]$.

When $W$ is the symmetric group $S_n$, the theory of cell modules shows that $\chi_\sf{v}(c_\z) \in \bb{Z}_{\geq 0}[\sf{v}^{\pm 1}]$ for all $\z \in W$ and $\chi \in \Irr(W)$, a positivity statement that already fails when $W$ is the Weyl group of type $BC_2$.
Haiman conjectured a much stronger positivity.
To state it, first recall that the elements of $\Irr(S_n)$ are indexed by integer partitions of $n$.
For any $\lambda \vdash n$, let $\chi_\lambda \in \Irr(S_n)$ be the corresponding irreducible character.
Next, recall that the Frobenius characteristic map~\cite[\S{1.7}]{macdonald} is a linear bijection between class functions on $S_n$ and symmetric functions of degree $n$ (say, in $n$ variables), taking $\chi_\lambda$ to the Schur function indexed by $\lambda$.

For any $\mu \vdash n$, let $\mon_\mu \colon S_n \to \bb{Z}$ be the class function whose Frobenius characteristic is the \emph{monomial} symmetric function indexed by $\mu$.
Then
\begin{align}
\chi_\lambda = \sum_{\mu \vdash n} K_{\lambda, \mu} \mon_\mu,
\end{align}
where $K_{\lambda, \mu} \in \bb{Z}_{\geq 0}$ is the Kostka number counting semistandard Young tableaux of shape $\lambda$ and weight $\mu$.
For any $\z \in S_n$, let
\begin{align}
\alpha_\mu^\z(\sf{v}) = \mon_{\mu, \sf{v}}(c_\z) \in \bb{Z}[\sf{v}^{\pm 1}].
\end{align}
What follows is Conjecture 2.1 in~\cite{haiman}.

\begin{conj}[Haiman 1993]\label[conj]{conj:haiman}
For any $\z \in S_n$ and $\mu \vdash n$, the nonzero coefficients of $\alpha_\mu^\z$ are all positive and form a unimodal sequence.
\end{conj}

Haiman points out that the coefficients of $\alpha_\mu^\z(\sf{v})$ are at least palindromic, due to the invariance of $c_\z$ under the anti-automorphism sending $\delta_s, \sf{v}$ to $\delta_s^{-1}, \sf{v}^{-1}$.

Further motivation for \Cref{conj:haiman} appeared in~\cite{chss}, where Clearman, Hyatt, Shelton, and Skandera showed that the positivity of $\mon_{\mu, \sf{v}}(c_\z)$ would subsume the Shareshian--Wachs conjecture: the $e$-positivity of the chromatic quasisymmetric functions of unit interval graphs~\cite{sw}.
Thus it would also subsume the Stanley--Stembridge conjecture~\cite{stanley,ss} recently proved in~\cite{hikita,gmrww}.

To explain the connection, let $\ind_\mu \colon S_n \to \bb{Z}$ be the class function whose Frobenius characteristic is the homogeneous symmetric function indexed by $\mu$.
Equivalently, $\ind_\mu$ is the character of the representation of $S_n$ induced from the parabolic subgroup $S_\mu \simeq S_{\mu_1} \times S_{\mu_2} \times \ldots$
Then 
\begin{align}
\ind_\mu = \sum_{\lambda \vdash n} K_{\lambda, \mu} \chi_\lambda.
\end{align}
So the Laurent polynomials $\alpha_\mu^\z(\sf{v})$ are uniquely determined by requiring that
\begin{align}\label{eq:haiman-iota}
\sum_{\lambda \vdash n} \chi_{\lambda, \sf{v}} (c_\z)\chi_\lambda = \sum_{\mu \vdash n} \alpha_\mu^\z(\sf{v}) \ind_\mu
\end{align}
as functions on $S_n$.

Each unit interval graph on $n$ vertices corresponds to some permutation $\z \in S_n$ that is \dfemph{$312$-avoiding},\footnote{In Haiman's terminology, \dfemph{codominant}.} meaning no indices $1 \leq i < j < k \leq n$ exist for which $z(j) < z(k) < z(i)$.
The authors of~\cite{chss} show that up to a sign twist and power of $\sf{v}$, the Frobenius characteristic of the left-hand side of \eqref{eq:haiman-iota} for this $\z$ is the chromatic quasisymmetric function for the graph.
In this way, \Cref{conj:haiman} would imply the positivity conjectured by Shareshian--Wachs.

For further details about unit interval graphs and their chromatic quasisymmetric functions, we defer to the references above.

\subsection{Lusztig Varieties}\label[subsection]{subsec:lusztig}

At the end of~\cite[\S{2}]{haiman}, Haiman writes:
\begin{quote} 
	There is probably an analog of Conjecture 2.1 for other Coxeter groups or at least Weyl groups. At present, however, it is not apparent what the analog of a monomial character should be.
\end{quote} 
Though we still lack a meaningful generalization of the characters $\mon_\mu$ beyond $S_n$, we will present a way to generalize \Cref{conj:haiman} to finite Weyl groups, using the work of Springer and Lusztig in algebraic geometry.

Already in~\cite{sw}, Shareshian--Wachs proposed an algebro-geometric formula\footnote{Another statement known as the Shareshian--Wachs conjecture.} for the chromatic quasisymmetric function of any unit interval graph, in terms of an $S_n$-action on the singular cohomology of a complex algebraic variety associated with the graph. 
In Lie-theoretic language, these varieties are Hessenberg varieties for the group $\GL_n$, depending not only on the graph, but on a regular semisimple element in the Lie algebra $\gl_n$.
The formula was proved by Brosnan--Chow~\cite{bc}.
In particular, Brosnan--Chow showed how to construct the $S_n$-action from the monodromy of a family of Hessenberg varieties over the regular semisimple locus of $\gl_n$.

Hessenberg varieties can be generalized to any reductive algebraic group $G$, by replacing unit interval graphs with so-called Hessenberg subspaces of the Lie algebra of $G$.
This suggests a generalization of the $e$-positivity conjecture in terms of Hessenberg varieties, replacing $S_n$ with the Weyl group of $G$.
More daringly, it suggests a generalization of \Cref{conj:haiman} in terms of varieties associated with elements of the Weyl group.

In~\cite{an_25}, Abreu--Nigro showed that the left-hand side of \eqref{eq:haiman-iota} records a monodromy representation on the \emph{intersection} cohomology of a variety $\bar{Y}_{\z, g}$ built from $\z$ and an arbitrary choice of regular semisimple element $g \in \GL_n(\bb{C})$.
This variety is now called a Lusztig variety, as it was first introduced by Lusztig in~\cite{lusztig_79}.
Explicitly, if $\cal{B}$ denotes the flag variety of $G$, and $W$ denotes its Weyl group, then for any $\z \in W$ and $g \in G(\bb{C})$, the \dfemph{closed Lusztig variety} $\bar{Y}_{\z, g}$ is the Zariski closure in $\cal{B}$ of the \dfemph{open Lusztig variety}
\begin{align}
Y_{\z, g} \vcentcolon= \{B \in \cal{B} \mid B \xrightarrow{\z} g^{-1}Bg\},
\end{align}
where $B \xrightarrow{\z} B'$ indicates that a pair of Borels $(B, B')$ is in relative position $\z$.
Intersection cohomology
%, rather than singular cohomology,
is needed because 
%in general,
$\bar{Y}_{\z, g}$ may be singular.
Abreu--Nigro's result harmonizes with recent work of Brosnan--Hong--Lee~\cite{bhl}, showing that when a (smooth) element $\z \in W$ corresponds to a Hessenberg subspace $\fr{H}$ of the Lie algebra of $G$, the $W$-representation arising from a Lusztig variety for $\z$ matches that arising from a Hessenberg variety for $\fr{H}$.

In this paper, we will use \'etale cohomology theories with $\ell$-adic coefficients.
Let $\ur{IH}^\ast$ denote intersection cohomology.
For a general connected complex reductive algebraic group $G$, we write $G^\rs$ to denote its regular semisimple locus, a Zariski dense open subset.
We are led to seek a generalization of \Cref{conj:haiman} that replaces $S_n$ with its Weyl group $W$ and the left-hand side of \eqref{eq:haiman-iota} with $\ur{IH}^\ast(\bar{Y}_{\z, g})$ for $g \in G^\rs(\bb{C})$.

It turns out that $\ur{IH}^i(\bar{Y}_{\z, g})$ vanishes for odd $i$ and its top degree is $\ell(\z)$, the Bruhat length of $\z$.
Moreover, by a result of Abreu--Nigro, the graded representation of $W$ on $\ur{IH}^\ast(\bar{Y}_{\z, g})$ is constant up to isomorphism as $g$ varies in $G^\rs(\bb{C})$.
Our first theorem is a purely algebraic formula for this representation in terms of a pairing $\{-, -\} \colon \Irr(W) \times \Irr(W) \to \bb{Q}$ called the \dfemph{(truncated) exotic Fourier transform}, originating in Lusztig's work on characters of finite reductive groups~\cite{lusztig_84}.
We review its definition in \Cref{sec:dl}.
For $W$ of classical type, Lusztig gives a purely combinatorial algorithm to calculate $\{-, -\}$.

For any $\z \in W$, let $\tau_{\z, \sf{v}} \colon W \to \bb{Z}[\sf{v}^{\pm 1}]$ be defined by
\begin{align}
	\tau_{\z, \sf{v}} = \sum_{\chi, \psi \in \Irr(W)} {\{\chi, \psi\}}\psi_\sf{v}(c_\z)\chi.
\end{align}
Note that $\{-, -\}$ and hence $\tau_{\z, \sf{v}}$ only depend on $W$, not on $G$.
In \Cref{sec:rs}, we use the work of Lusztig and Abreu--Nigro to prove the following formula.

\begin{thm}\label[thm]{thm:rs}
Fix a connected complex reductive group $G$ with Weyl group $W$.
Then for any $\z \in W$ and $g \in G^\rs(\bb{C})$ and $w \in W$, we have
\begin{align}
\sum_i {\sf{v}^{2i}} \tr(w \mid \ur{IH}^{2i}(\bar{Y}_{\z, g}))
	= \sf{v}^{\ell(\z)} \tau_{\z, \sf{v}}(w),
\end{align}
where the $W$-action on the intersection cohomology arises from the monodromy action of the fundamental group of $G^\rs$ based at $g$.
\end{thm}

\begin{ex}\label[ex]{ex:kronecker}
In type $A$, where $W = S_n$, the exotic Fourier transform is simply the pairing where $\{\chi, \chi\} = 1$ and $\{\chi, \psi\} = 0$ for $\chi \neq \psi$.
So here, $\tau_{\z, \sf{v}}$ recovers the left-hand side of \eqref{eq:haiman-iota}, and \Cref{thm:rs} recovers~\cite[Theorem~1.5]{an_25}.
\end{ex} 

The linear function on $H_W$ that sends $c_\z \mapsto \tau_{\z, \sf{v}}$, being a linear combination of the characters $\psi_\sf{v}$, may be viewed as a trace on $H_W$.
It appeared in~\cite{trinh}, in relation to Lusztig varieties over the \emph{unipotent} locus $G^\unip \subseteq G$.
Pursuing the viewpoint there yields a parallel interpretation of $\tau_{\z, \sf{v}}$ in terms of $G^\unip$ rather than $G^\rs$, as we now explain.
%Patrick Brosnan has pointed out that this is reminiscent of how the regular semisimple and unipotent loci of the Lie algebra of $G$ are related under the Fourier--Deligne transform.

Recall that in Lusztig's work, the varieties $Y_{\z, g}$, \emph{resp.}\ $\bar{Y}_{\z, g}$, are the fibers of maps $Y_\z \to G$, \emph{resp.}\ $\bar{Y}_\z \to G$.
For the identity element $\z = e$, we have $\bar{Y}_e = Y_e$.
The restricted map $Y_e^\unip \to G^\unip$ is a resolution of singularities of $G^\unip$ now called the \dfemph{Springer resolution}.
Springer discovered a remarkable action of $W$ on the cohomologies of its fibers, not induced by an action on the fibers themselves.
More precisely, Springer worked with an analogue for the Lie algebra of $G$, and over a finite field, not over $\bb{C}$.
The Fourier--Deligne transform intertwines his action with a monodromy action over the regular semisimple locus of the Lie algebra.

Up to a sign twist, the action of $W$ can also be constructed from convolution diagrams involving the \dfemph{Steinberg variety}
\begin{align}
\cal{Z} = Y_e^\unip \times_{G^\unip} Y_e^\unip.
\end{align}
Convolution turns the Borel--Moore homology of $\cal{Z}$ into a graded algebra that acts on the cohomology of any Springer fiber.
Motivated by~\cite{trinh}, we introduce
\begin{align}\label{eq:steinberg}
\bar{\cal{Z}}_\z 
	&= Y_e^\unip \times_G \bar{Y}_\z.
\end{align}
The Borel--Moore homology $\ur{H}_\ast^\ur{BM}(\cal{Z})$ then acts by convolution on the compactly-supported cohomology $\ur{H}_c^\ast(\bar{\cal{Z}}_\z)$.

As in~\cite{trinh}, it is more convenient to work $G$-equivariantly.
We have $\ur{H}_\ast^{\ur{BM}, G}(\cal{Z}) \simeq \bb{C}W \otimes \Sym^\ast(\mathbb{X})$, where $\mathbb{X}$ denotes the character lattice in the root datum of $G$, on which $W$ acts as a reflection group, and $\Sym^\ast(\mathbb{X})$ denotes the symmetric algebra generated by $\mathbb{X}$ in degree $2$.
Let $\mathbb{S}_\sf{q} \colon W \to \bb{Z}[\![\sf{q}]\!]^\times$ be the graded character 
\begin{align}
\mathbb{S}_\sf{q}(w) \vcentcolon= \sum_i \sf{q}^i \tr(w \mid \Sym^i(\mathbb{X})) = \frac{1}{\det(1 - \sf{q}w \mid \mathbb{X})}.
\end{align}
With this notation, we can state a unipotent counterpart to \Cref{thm:rs}.
In what follows, recall that $\z$ is \dfemph{rationally smooth} if and only if the Kazhdan--Lusztig polynomial of $(e, \z)$ is trivial~\cite{kl}.

\begin{thm}\label[thm]{thm:unipotent}
Fix a connected complex reductive group $G$ with Weyl group $W$.
If $\z \in W$ is rationally smooth, then for all $w \in W$, we have
\begin{align}
\sum_j {(-\sf{v})^j} \tr(w \mid \ur{H}_{c, G}^j(\bar{\cal{Z}}_\z))
	= (-1)^{\ell(w)} \sf{v}^{2\dim G + \ell(\z)}
		\mathbb{S}_{\sf{v}^2}(w) \tau_{\z, \sf{v}}(w),
\end{align}
where the $W$-action on the equivariant compactly-supported cohomology arises from convolution. 
%In particular, the latter vanishes in odd degrees.
\end{thm} 

In \Cref{sec:unipotent}, we prove this result, and sketch a generalization that applies to all $\z$, not just rationally smooth $\z$.

We mentioned earlier that the Shareshian--Wachs symmetric function for a unit interval graph is essentially the Frobenius characteristic of $\tau_{\z, \sf{v}}$ for some $312$-avoiding element $\z \in S_n$.
Proposition 3.1 of~\cite{haiman} shows that if $\z$ is a $312$-avoiding, then its Schubert variety is smooth, so $\z$ is rationally smooth.

When $\mathbb{X}$ has rank $n$, Frobenius characteristic transports the pointwise product with $\mathbb{S}_\sf{q}$ to the plethystic transform $X_i \mapsto X_i/(1 - \sf{q})$ in the symmetric-function variables $X_1, X_2, \ldots$
The plethystic transforms of the Shareshian--Wachs functions appear in the work of Carlsson--Mellit on the shuffle conjecture, where they are called \emph{characteristic functions of Dyck paths}~\cite{cm}.
As explained in~\cite{ap}, they are precisely the Lascoux--Leclerc--Thibon (LLT) polynomials for tuples of shifted partitions in which each partition has size $\leq 1$.
Hence, \Cref{thm:unipotent} provides a new algebro-geometric interpretation of these \emph{unicellular} LLT polynomials.

\begin{rem}
In~\cite{kato}, S.~Kato geometrizes Shareshian--Wachs functions in terms of varieties over the affine Grassmannian of $\GL_n$.
It would be interesting to know whether the map $\bar{\cal{Z}}_\z \to G^\unip$, for $312$-avoiding $\z \in S_n$, is related to Kato's work: say, via Lusztig's embedding of the nilpotent cone into the affine Grassmannian.
\end{rem} 

\subsection{The Springer Correspondence}\label[subsection]{subsec:springer}

Returning to \Cref{conj:haiman}, we next seek to generalize the characters $\ind_\mu$ on the right-hand side of \eqref{eq:haiman-iota}.
It turns out that simply replacing them with the characters of $W$ induced from trivial characters of parabolic subgroups does not work,
% even in type $B_2$,
because these are too few.
Fortuitously, the work of Springer discussed above offers another generalization.

In what follows, for all $u \in G^\unip(\bb{C})$, we replace the notation $Y_{e, u}$ with the more standard notation $\cal{B}_u$.
Explicitly,
\begin{align}
\cal{B}_u = \{B \in \cal{B} \mid u \in B\}.
\end{align}
%Recall that $W$ acts on the cohomology of the Springer fiber $\cal{B}_u$.
A purity result shows that $\ur{H}^i(\cal{B}_u)$, like $\ur{IH}^\ast(\bar{Y}_{\z, g})$, vanishes for odd $i$.

Let $C_u \subseteq G$ be the conjugacy class of $u$ and $G_u \subseteq G$ its centralizer.
Via the isomorphism of varieties $C_u \simeq G/G_u$, any representation of the component group $A_u \vcentcolon= \pi_0(G_u)$ defines a local system on $C_u$.
The cohomology of $\cal{B}_u$ decomposes into pieces indexed by such local systems.
Indeed, the $G$-action on $\cal{B}$ restricts to a $G_u$-action on $\cal{B}_u$, which induces an $A_u$-action on $\ur{H}^\ast(\cal{B}_u)$.
The actions of $W$ and $A_u$ on $\ur{H}^\ast(\cal{B}_u)$ centralize each other.
For any $\kappa \in \Irr(A_u)$, we write $\ur{H}^\ast(\cal{B}_u)_\kappa$ to denote the multiplicity space of $\kappa$ in $\ur{H}^\ast(\cal{B}_u)$.

%Let $d_u$ be the dimension of $\cal{B}_u$.
The \dfemph{Springer correspondence} states that each irreducible character $\psi \in \Irr(W)$ is afforded by the $W$-action on $\ur{H}^{2\dim \cal{B}_u}(\cal{B}_u)_\kappa$ for some unipotent $u$ and $\kappa \in \Irr(A_u)$, and that the pair $(u, \kappa)$ is uniquely determined up to conjugacy by $\psi$.
%(The vector space $\ur{H}^{2d_u}(\cal{B}_u)_\kappa$ vanishes for other $(u, \kappa)$.)
%, in which case the orbit $[u, \kappa]$ is said to be \dfemph{cuspidal}.
%Lusztig showed~\cite{lusztig_84_intersection} that there are very few cuspidal local systems:
%If $G$ is almost-simple and adjoint, then there is at most one, and only in types $B$, $C$, $D$, $G_2$, $F_4$, $E_8$.
For such $\psi$ and $(u, \kappa)$, let $\spr_{\psi, G} = \spr_{u, \kappa} \colon W \to \bb{Z}$ be defined by
\begin{align}\label{eq:spr}
\spr_{\psi, G}(w) = \spr_{u, \kappa}(w) \vcentcolon= \sum_i \tr(w \mid \ur{H}^{2i}(\cal{B}_u)_\kappa).
\end{align}
These functions depend on the root system of $G$, though not on the isomorphism class of $G$ itself.
In type $A$, if $G$ has connected center, then the groups $A_u$ are trivial, and $\spr_{\chi_\mu, G}$ recovers $\ind_\mu$.
In other classical types, D.~Kim gives a purely combinatorial formula for a coarser function, the character of the total cohomology $\ur{H}^\ast(\cal{B}_u)$, in terms of modified Kostka polynomials~\cite{kim}.
In all types, the linear operator on class functions that sends $\psi \mapsto \spr_{\psi, G}$ is unipotent, hence invertible.

For any $\z \in W$, let $\alpha_{\psi, G}^\z(\sf{v}) \in \bb{Z}[\sf{v}^{\pm 1}]$ be the (necessarily palindromic) Laurent polynomials indexed by $\psi \in \Irr(W)$ such that
\begin{align}\label{eq:main}
\tau_{\z, \sf{v}}
	= \sum_{\psi \in \Irr(W)} \alpha_{\psi, G}^\z(\sf{v}) \spr_{\psi, G}
\end{align}
as functions on $W$.
%\Cref{thm:rs} shows that these virtual multiplicities determine and are determined by the monodromy of the closed Lusztig varieties $\bar{Y}_{\z, g}$ for regular semisimple $g \in G^\rs(\bb{C})$.
It would be pleasant to generalize \Cref{conj:haiman} by generalizing Haiman's $\alpha_\mu^\z$ to our $\alpha_{\psi, G}^\z$.
Unfortunately, the nonzero coefficients of $\alpha_{\psi, G}^\z$ can fail to be positive.
In many cases, they are all negative, but in rare cases, they have mixed signs.
They can also fail to form a unimodal sequence, even in absolute value.

\begin{rem}
A result from the character theory of finite reductive groups, the \emph{orthogonality of Green functions}, implies that up to a polynomial denominator, the function $\sf{v}^{\ell(\z)} \mathbb{S}_{\sf{v}^2} \cdot \tau_{\z, \sf{v}}$ in \Cref{thm:unipotent} is a $\bb{Z}[\sf{v}^2]$-linear combination of the functions $\spr_{u, a, \sf{v}^2} \colon W \to \bb{Z}[\sf{v}^2]$ defined by
\begin{align} 
\spr_{u, a, \sf{v}^2}(w)
	= \sum_i \sf{v}^{2i}
		\sum_{\kappa \in \Irr(A_u)} \kappa(a) \tr(w \mid \ur{H}^{2i}(\cal{B}_u)_\kappa),
\end{align}
as we run over $u \in G^\unip(\bb{C})$ and $a \in A_u$.
Elsewhere, we will discuss a geometric interpretation of the coefficients in this expansion, which superficially resembles \eqref{eq:main}.
But so far, we do not know any direct relationship between the two expansions.
\end{rem} 

We introduce the following shorthand for the properties that interest us:
%\begin{array}{r@{\::\:}l}
\begin{align}
(\vartriangle)_{\psi, G}^\z \colon 
	&\text{the nonzero coefficients of $\alpha_{\psi, G}^\z$ are unimodal in absolute value},\\
(\pm)_{\psi, G}^\z \colon 
	&\text{the nonzero coefficients of $\alpha_{\psi, G}^\z$ are either all positive or all negative},\\
(+)_{\psi, G}^\z \colon 
	&\text{the nonzero coefficients of $\alpha_{\psi, G}^\z$ are all positive}.
\end{align}
The examples in \Cref{sec:examples} lead us to believe that as the rank of $W$ grows, the elements $\z$ where $(\vartriangle)_{\psi, G}^\z$ and $(\pm)_{\psi, G}^\z$ hold for all $\psi$, but not $(+)_{\psi, G}^\z$, come to predominate in $W$.
%The failures occur only when $\psi$ belongs to a fixed, relatively small subset $\Irr(W)' \subseteq \Irr(W)$.
%The elements of $\Irr(W)'$ tend to correspond, under Springer, to local systems on unipotent conjugacy classes whose dimension is either relatively large or relatively small.
%For instance, the characters of the reflection representation and its Lusztig--Spaltenstein dual~\cite[\S{13.4}]{carter} often belong to $\Irr(W)'$.
%The former maps to the trivial local system on the subregular class, the latter to the trivial local system on the minimal nontrivial class.
At the same time, we believe that the irreducible characters $\psi$ where $(\vartriangle)_{\psi, G}^\z$ and $(\pm)_{\psi, G}^\z$ hold for all $\z$, but not $(+)_{\psi, G}^\z$, come to predominate in $\Irr(W)$.

\begin{quest}
Given $\omega \in \{\vartriangle, \pm, +\}$:
\begin{enumerate} 
\item[(I)]
	What properties of $\z \in W$ ensure that $(\omega)_{\psi, G}^\z$ holds for all $\psi$?

\item[(II)]
	What properties of $\psi \in \Irr(W)$ ensure that $(\omega)_{\psi, G}^\z$ holds for all $\z$?

\end{enumerate}
\end{quest} 

\begin{rem} 
The truth of these properties does depend on the root system of $G$, not just on $W$:
In the examples in \Cref{sec:examples}, they fail for more pairs $(\z, \psi)$ in type $C_n$ than in type $B_n$.
%In general, the respective subsets of $W \times \Irr(W)$ where the failures occur are incomparable as subsets.
\end{rem} 

\begin{rem}\label[rem]{rem:multiplicative}
Suppose that $G$ is isogenous to $G_1 \times G_2$.
Let $W_i$ be the Weyl group of $G_i$, so that $W \simeq W_1 \times W_2$.
Then the pairing $\{-, -\}$ and the functions $\psi \mapsto \spr_{\psi, G}$ are multiplicative with respect to the bijection $\Irr(W) \simeq \Irr(W_1) \times \Irr(W_2)$.
We deduce that if $(\omega)_{\psi_1, G_1}^{\z_1}$ and $(\omega)_{\psi_2, G_2}^{\z_2}$ hold for some $\z_i \in W_i$ and $\psi_i \in \Irr(W_i)$, then $(\omega)_{\psi, G}^\z$ holds for $\z = (\z_1, \z_2)$ and $\psi = \psi_1 \times \psi_2$.

In this way, we can reduce the study of $(\omega)_{\psi, G}^\z$ to the cases where $W$ is irreducible as a Coxeter group.
In particular, we can reduce the cases where $W$ is a product of symmetric groups and $\omega \in \{\vartriangle, +\}$ to \Cref{conj:haiman}.
\end{rem} 

To answer question (I), we prove results that directly generalize the three parts of~\cite[Proposition~4.1]{haiman}.

In~\cite{trinh}, we proved a result essentially stating that $\tau_{\z, \sf{v}}$ intertwines parabolic inclusions $W' \subseteq W$ with induction of characters from $W'$ to $W$.
In \Cref{sec:parabolic}, we prove that the characters $\spr_{\psi, G}$ enjoy a similar compatibility with induction (\Cref{thm:ind}), refining an observation of Alvis--Lusztig: 
This may be of independent interest.
From combining these ``induction theorems'', we obtain the result below: in fact, a more precise form generalizing part (2) of~\cite[Proposition~4.1]{haiman}.

\begin{prop}\label[prop]{prop:parabolic}
If $\z$ lies in a parabolic subgroup $W' \subseteq W$, corresponding to a Levi subgroup $G' \subseteq G$, then the $\alpha_{\psi, G}^\z$ for $\psi \in \Irr(W)$ are nonnegative integer sums of the $\alpha_{\psi', G'}^\z$ for $\psi' \in \Irr(W')$.
In particular:
If $(+)_{\psi', G'}^\z$ holds for all $\psi'$, then $(+)_{\psi, G}^\z$ holds for all $\psi$.
If $(\vartriangle)_{\psi', G'}^\z$ and $(+)_{\psi', G'}^\z$ hold for all $\psi'$, then $(\vartriangle)_{\psi, G}^\z$ and $(+)_{\psi, G}^\z$ hold for all $\psi$.
\end{prop}

%By \Cref{rem:multiplicative}, \Cref{conj:haiman} would imply $(\vartriangle)_{\psi, G}^\z$ and $(+)_{\psi, G}^\z$ for all $\z$ and $\psi$ in all cases where $W$ is a product of symmetric groups.
%We deduce:

%\begin{cor}
%If \Cref{conj:haiman} holds, then $(\vartriangle)_{\psi, G}^\z$ and $(+)_{\psi, G}^\z$ hold for any $G$ with Weyl group $W$, any $\z$ in a parabolic subgroup of $W$ isomorphic to a product of symmetric groups, and any $\psi \in \Irr(W)$.
%\end{cor}

We can also give a sufficient condition for the $\alpha_{\psi, G}^\z$ to vanish.
Recall that the Kazhdan--Lusztig basis defines a partition of $W$ into subsets $\cal{F}$ called \dfemph{two-sided cells}, arranged in a partial order.
Each $\cal{F}$ gives rise to a module over the Hecke algebra, and each simple Hecke module occurs in exactly one such module, producing a map $\chi \mapsto \cal{F}(\chi)$ from irreducible characters of $W$ to two-sided cells.
When $W = S_n$, this map is a bijection, and the partial order on two-sided cells recovers the reverse dominance order on partitions of $n$.
So the following result, explained in \Cref{sec:cells}, generalizes part (3) of~\cite[Proposition~4.1]{haiman}.

\begin{prop}\label[prop]{prop:cells}
If $\psi \in \Irr(W)$ and $\z \in \cal{F}$ for some two-sided cell $\cal{F} < \cal{F}(\psi)$, then $\alpha_{\psi, G}^\z(\sf{v}) = 0$.
\end{prop}

Let $w_\circ$ be the longest element of $W$.
The following result generalizes part (1) of~\cite[Proposition~4.1]{haiman}.

\begin{cor}\label[cor]{cor:cells-bottom}
For the trivial character $1_W$, we have
\begin{align}
\alpha_{1_W, G}^{w_\circ}(\sf{v}) = \displaystyle\sum_{w \in W} \sf{v}^{2\ell(w) - \ell(w_\circ)}.
\end{align}
If $\psi$ is nontrivial, then $\alpha_{\psi, G}^{w_\circ}(\sf{v}) = 0$.
\end{cor}

It seems tricky to formulate other uniform results about question (I).
For example, we do not know how to generalize~\cite[Proposition~4.2]{haiman}, which concerns the case where $\z$ is a Coxeter element:
For each Coxeter element $\z$ in type $D_4$, it turns out that $(+)_{\psi, G}^\z$ fails when $\psi$ is the character labeled $1.3$ in CHEVIE.

As another example:
At an earlier stage of this work, we had conjectured that if $W = W(B_n) = W(C_n)$ for some $n$ and $\z$ is rationally smooth, then $(\vartriangle)_{\psi, G}^\z$ and $(\pm)_{\psi, G}^\z$ hold for all $\psi$.
This is true up through rank $6$, but both properties fail in both type $B_7$ and type $C_7$ for many rationally smooth $\z$.
In the cases where $(\pm)_{\psi, G}^\z$ fails, $\ell(\z) \geq 18$.

\subsection{A New Conjecture}\label[subsection]{subsec:new}

Question (II) led us to \Cref{conj:main} below.
It suggests that some compatibility between inflation maps and the Springer correspondence 
allows the positivity and unimodality in Haiman's \Cref{conj:haiman} to be lifted to other Weyl groups.

Observe that there is a unique quotient map $W(G_2) \to W(A_2) = S_3$ sending simple reflections to simple reflections.
We call it the \dfemph{$G_2$-to-$A_2$ quotient}.

\begin{conj}\label[conj]{conj:main}
Suppose that $\psi \in \Irr(W)$ is inflated from a quotient $\varphi \colon W \to \Omega$.
If $\Omega$ is isomorphic to a product of symmetric groups, then $(\vartriangle)_{\psi, G}^\z$ and $(\pm)_{\psi, G}^\z$ hold for all $z \in W$.
Moreover, if $\varphi$ does not descend to the $G_2$-to-$A_2$ quotient, then $(+)_{\psi, G}^\z$ holds for all $\z$.
\end{conj} 

\begin{ex}
The trivial and sign characters of $W$ are inflated from $S_2 \simeq \{\pm 1\}$ along the sign character itself.
So \Cref{conj:main} predicts that if $\psi$ is either of these characters, then $(\vartriangle)_{\psi, G}^\z$ and $(+)_{\psi, G}^\z$ hold for all $\z$.

For the trivial character, we do not know a proof of these statements.
Compare to~\cite[Conjecture~4.1]{haiman}.

For the sign character $\varepsilon$, we claim that $\alpha_{\varepsilon, G}^e = 1$ and $\alpha_{\varepsilon, G}^\z = 0$ for all $\z \neq e$.
The latter is part (1) of \Cref{cor:cells}.
As for the former, \Cref{thm:rs} and the discussion in \S\ref{subsec:grothendieck} show that $\tau_{e, \sf{v}}$ is the character of the regular representation of $W$, which coincides with $\spr_{\varepsilon, G}$.
Thus, $\tau_{e, \sf{v}} = \spr_{\varepsilon, G}$, giving $\alpha_{\varepsilon, G}^e = 1$.
%This verifies \Cref{conj:main} for $\psi = \varepsilon$.
\end{ex} 

%When $W$ is irreducible, it is easier to classify quotients of $W$ than one might expect.
Maxwell showed~\cite{maxwell} that if $W$ is an irreducible finite Coxeter group and $\varphi \colon W \to \Omega$ is a quotient map, then either or both of the following must hold:
\begin{itemize} 
\item 	The longest element $w_\circ \in W$ is central, and $\varphi$ is reduction modulo $w_\circ$.
\item 	$\Omega$ is another Coxeter group, and $\varphi$ is a \dfemph{graph homomorphism}\footnote{So named because the graph homomorphisms from $(W, S)$ onto itself are the automorphisms induced by symmetries of its Dynkin diagram.}: that is, a quotient map that sends each simple reflection of $W$ either to a simple reflection of $\Omega$ or to the identity of $\Omega$.

\end{itemize}

\begin{ex}\label[ex]{ex:bc}
Take $W = W(B_n) = W(C_n)$.
Label the simple reflections of $W$ by $s_1, s_2, \ldots, s_n$, where $(s_1s_2)^4 = e$ and $(s_is_{i + 1})^3 = e$ for $i \geq 2$.
Then reduction modulo $ts_1ts_1$ is a graph homomorphism $\varphi_{1212} \colon W \to W(A_1 \times A_{n - 1}) = S_2 \times S_n$.

Following~\cite{carter}, we index the irreducible characters of $W$ by bipartitions of $n$: ordered pairs of partitions $(\xi, \eta)$ such that $n = |\xi| + |\eta|$.
Writing $\psi_{\xi, \eta}$ for the associated character, one checks that the following conditions are equivalent:
\begin{itemize} 
	\item	$\psi_{\xi, \eta}$ is inflated from $S_2 \times S_n$ along $\varphi_{1212}$.

	\item 	$\psi_{\xi, \eta}$ is inflated from a product of symmetric groups.
	
	\item 	One of $\xi$ or $\eta$ is empty.
	
\end{itemize}
Examining the third condition is how we discovered \Cref{conj:main}.

For instance, when $n = 4$, there are $20$ irreducible characters total. 
Of these, $10$ are inflated along $\varphi_{1212}$.
Of the remainder, there are $3$ where $(+)_{\psi, G}^\z$ fails in type $B_4$ for some $\z$, and $4$ where $(+)_{\psi, G}^\z$ fails in type $C_4$ for some $\z$.
\end{ex} 

%\begin{ex} 
%In type $F_4$, there are $25$ irreducible characters total.
%Of these, $9$ are inflated from the unique maximal quotient map to a product of symmetric groups.
%Of the remainder, there are $10$ where $(+)_{\psi, G}^\z$ fails for some $\z$.
%\end{ex} 

By \Cref{rem:multiplicative}, it is enough to check \Cref{conj:main} in the cases where $W$ is irreducible.
In \Cref{sec:examples}, we summarize the verification of \Cref{conj:main} for the following cases, using the GAP3 package CHEVIE~\cite{ghmp}:
\begin{enumerate}
\item 	$W$ irreducible of rank $\leq 5$.
\item 	$W$ irreducible of rank $\leq 6$, where we restrict $\z$ to be rationally smooth.
\end{enumerate}
As explained there, the data and scripts used in the verification are available at a webpage that also contains partial data in types $B_7$, $C_7$, $D_7$.
%In type $E$, \Cref{conj:main} is trivial because the Weyl groups contain simple groups as index-$2$ subgroups.

%When $W$ is large, the subset of rationally smooth elements is significantly smaller than $W$, yet still includes elements of substantial complexity, like the longest element.
%It is easier to test rational smoothness than smoothness, thanks to~\cite[Theorem~C]{carrell}.

There are other predictions that one could draw from the evidence in \Cref{sec:examples}, but perhaps none as striking as \Cref{conj:main}.
For example:
Given $\omega \in \{\vartriangle, \pm, +\}$, let $\Irr(W)_{G, \omega}$ be the set of $\psi \in \Irr(W)$ where $(\omega)_{\psi, G}^\z$ holds for all $\z \in W$.
In our examples, the sets $\Irr(W)_{G, \vartriangle}$ and $\Irr(W)_{G, \pm}$ are often equal, but not always; one may contain the other as a strict subset.
We have not found any cases where they are incomparable.
Moreover, they are noticeably larger than $\Irr(W)_{G, +}$.
It is necessarily true that $\Irr(W)_{G, +} \subseteq \Irr(W)_{G, \pm}$, but it is surprising to find that we always have $\Irr(W)_{G, +} \subseteq \Irr(W)_{G, \vartriangle}$ as well.
We expect that if $\psi$ is fixed, and $(\vartriangle)_{\psi, G}^\z$ or $(\pm)_{\psi, G}^\z$ fails for some $\z \in W$, then there is always some $\z' \in W$ where the nonzero coefficients of $\alpha_{\psi, G}^{\z'}$ are all \emph{negative}.

\subsection{Further Directions}

There are at least two.

\subsubsection{From Finite to Affine}

Affinization is the principle that structures for a complex reductive group $G$ often have analogues for its loop group $LG$ or the associated affine Kac--Moody group.

The affine analogue of a Lusztig variety was first studied in~\cite{lusztig_11_partitions} and~\cite{kv}.
The term \dfemph{affine Lusztig variety} was introduced by X.~He~\cite{he}, although other names have been used, such as \emph{generalized affine Springer fiber}, \emph{multiplicative affine Springer fiber}, and \emph{Kottwitz--Viehmann variety}.
Just as Lusztig varieties originally appear in the theory of character sheaves, so affine Lusztig varieties are expected to appear in a putative theory of character sheaves on loop groups.

Since affine Lusztig varieties for regular semisimple elements of the loop group $LG$ are still finite-dimensional, we can use them to define an analogue of $\tau_{\z, \sf{v}}$ geometrically.
It should take the form of a class function on an extended affine Weyl group.
However, it is less clear what should replace the functions $\spr_{\psi, G}$, since affine Springer fibers for unipotent elements of $LG$ are infinite-dimensional.
It seems plausible that one should use the ungraded characters of affine Springer representations arising from \emph{topologically} unipotent elements of $LG$.
Here, the analogues of the groups $A_u$ and their characters $\kappa$ play a major role in the arithmetic of the affine Springer fibers: more precisely, in the theory of endoscopy in the Langlands program~\cite{wang}.

\begin{quest}
Is there an analogue of~\Cref{conj:main}, relating the intersection cohomology of affine Lusztig varieties to ungraded affine Springer characters?
\end{quest}

\subsubsection{From Weyl Groups to Coxeter Groups}\label[subsubsection]{subsubsec:coxeter}

In~\cite{lusztig_93, lusztig_94}, Lusztig and Malle showed that certain postulates uniquely determine an analogue of the exotic Fourier transform from~\cite{lusztig_84} for any finite Coxeter group. 
(More accurately, Malle's appendix to~\cite{lusztig_94} handles type $H_4$.)
Consequently, $\tau_{\z, \sf{v}}$ can be generalized from Weyl groups to all finite groups.

At the same time, in~\cite{aa}, Achar--Aubert showed that certain axioms determine what they call \emph{preferred} analogues of the map $\psi \mapsto \spr_{\psi, G}$ for dihedral groups.
In~\cite{achar_rims}, Achar reports on attempts to extend the characters $\spr_{\psi, G}$ even further, to complex reflection groups.

\begin{quest} 
Can we uniformly generalize \Cref{conj:main} to finite Coxeter groups?
\end{quest} 

\subsection{Acknowledgments}

I thank Patrick Brosnan, Jaehyun Hong, Donggun Lee, and Eunjeong Lee for a discussion held shortly after I had sent them an early draft of this paper.
Independently, they too had observed \Cref{thm:rs}, and had used GAP3 and Sage to compute the Laurent polynomials $\alpha_{\psi, G}^\z(\sf{v})$ in low rank.
Moreover, in the draft I had sent, I had only claimed \Cref{thm:rs} for $g$ in a dense open locus of $G^\rs$.
I am grateful to them for pointing out that Abreu--Nigro's work allows us to extend that locus to $G^\rs$ itself.

I thank Roman Bezrukavnikov and Eric Sommers for discussions that helped me to arrive at \Cref{thm:ind} and \Cref{lem:inj}, respectively.

I thank Claude (Sonnet 4.6) for helping me to write Python scripts to analyze CHEVIE data.
I thank Rachel McEnroe and Martha Precup for their encouragement, and Nathan Williams for steering me toward more robust conjectures.

\section{Setup}\label[section]{sec:setup}

This section reviews the geometry needed in \Cref{sec:rs,sec:unipotent}.

Note that, while \Cref{sec:intro} discussed algebraic varieties over the complex numbers, Lusztig's papers on character sheaves start from a general algebraically closed field $\bb{k}$.
We do the same, and specialize $\bb{k}$ to the field of complex numbers or to the algebraic closure of a finite field only when needed.
We fix a prime $\ell$ invertible in $\bb{k}$, and work with $\QL$-complexes in the \'etale topology on $\bb{k}$-varieties.

Let $G$ be a connected reductive algebraic group over $\bb{k}$ with Weyl group $W$.
The characteristic of $\bb{k}$ is either zero or a good prime for $G$~\cite[28]{carter}.

\subsection{Unipotent Character Sheaves}

Let $\cal{B}$ be the flag variety of $G$.
Let 
\begin{align} 
\act \colon \cal{B} \times G \to \cal{B} \times \cal{B}
\end{align} 
be defined by $\act(B, g) = (B, g^{-1}Bg)$.
For all $\z \in W$, let $O_\z \subseteq \cal{B} \times \cal{B}$ be the $G$-orbit of pairs in relative position $\z$.
Sheaves on these orbits form a geometrization of the Hecke algebra $H_W$, while pullback-pushforward through the correspondence
\begin{align}
	\cal{B} \times \cal{B} \xleftarrow{\act} \cal{B} \times G \xrightarrow{\pr} G
\end{align}
is a geometrization of its cocenter map.

Our notations will mostly follow Lusztig's notations in~\cite{lusztig_85_3}.
Let $Y_\z \subseteq \cal{B} \times G$ be defined by the pullback square:
\begin{equation}
	\begin{tikzcd}
		O_\z \arrow{d} &Y_\z \arrow{l} \arrow{d}\\
		\cal{B} \times \cal{B} &\cal{B} \times G \arrow{l}[above]{\act}
	\end{tikzcd}
\end{equation}
%The map $Y_\z \to O_\z$ is smooth of relative dimension $\dim G - \dim \cal{B}$.
%Thus $Y_\z$ is smooth of pure dimension $\dim G + \ell(\z)$.
%
Let $\bar{Y}_\z$ be the Zariski closure of $Y_\z$ in $\cal{B} \times G$.
Thus $\bar{Y}_\z = \bigcup_{w \leq v} Y_w$, where $\leq$ is the Bruhat order on $W$ corresponding to the closure order on the orbits $O_w$.

Let $j \colon Y_\z \to \bar{Y}_\z$ be the inclusion map.
Let $\pi_\z \colon Y_\z \to G$ and $\bar{\pi}_\z \colon \bar{Y}_\z \to G$ be the projection maps that factor through $\pr$.

Let $j_{!\ast}$, denote intermediate extension of shifted perverse sheaves along $j$~\cite[\S{5.1}]{achar}.
Let $J_\z = j_{!\ast}\underline{\QL}_{Y_\z}$, the (unshifted) intersection cohomology complex of $\bar{Y}_\z$.
Let $\bar{K}_\z = \bar{\pi}_{\z, !}J_\z = \bar{\pi}_{\z, \ast} J_\z$.
By the decomposition theorem~\cite[\S{3.9}]{achar}, there is a noncanonical isomorphism
\begin{align}\label{eq:decomposition-thm}
	\bar{K}_\z \simeq \bigoplus_i {}^p\cal{H}^i(\bar{K}_\z)[-i],
\end{align}
where ${}^p\cal{H}^i$ denotes the $i$th cohomology sheaf with respect to the perverse $t$-structure, and each perverse sheaf ${}^p\cal{H}^i(\bar{K}_\z)$ is semisimple.
The simple summands of these perverse sheaves, as we run over $\z$ and $i$, are called \dfemph{unipotent character sheaves}.

\subsection{The Grothendieck Alteration}\label[subsection]{subsec:grothendieck}

The map $\pi_e = \bar{\pi}_e$ is called the \dfemph{Grothendieck alteration} of $G$.
Below, we review facts about $\pi_e$ that can be proved in an analogous way to their counterparts for the Lie algebra of $G$~\cite[\S{8.2}]{achar}.

First, the pullback of $\pi_e$ to the regular semisimple locus $G^\rs$ is a Galois \'etale cover with Galois group $W$.
In particular, if $\bb{k} = \bb{C}$, then $W$ is a quotient of the topological fundamental group of $G^\rs$.

Next, $\pi_e$ is a small map.
So $\bar{K}_e$ is the intermediate extension of its restriction $\bar{K}_e|_{G^\rs}$, and the Galois action of $W$ on $\bar{K}_e|_{G^\rs}$ induces one on $\bar{K}_e$ by functoriality.
We obtain a $W$-equivariant decomposition
\begin{align}\label{eq:grothendieck}
\bar{K}_e[\dim G] \simeq \bigoplus_{\chi \in \Irr(W)} V_\chi \otimes \cal{A}_\chi,
\end{align}
where the $\cal{A}_\chi$ are simple perverse sheaves and $V_\chi \vcentcolon= \Hom(\cal{A}_\chi, \bar{K}_e[\dim G])$ affords a representation of $W$ with character $\chi$.

A unipotent character sheaf that is not isomorphic to $\cal{A}_\chi$ for any $\chi$ is called \dfemph{cuspidal}~\cite{lusztig_85_2}.

\begin{thm}[Lusztig]\label[thm]{thm:exotic}
Suppose that $\bb{k}$ is either of characteristic zero or the algebraic closure of a finite field.
Then for any $\z \in W$ and $\chi \in \Irr(W)$, we have
\begin{align}
	\sum_i {(-\sf{v})^i} (\cal{A}_\chi : {}^p\cal{H}^i(\bar{K}_\z))
	= (-1)^{\dim G} \sf{v}^{\dim G + \ell(\z)} \sum_{\psi \in \Irr(W)}
	{\{\chi, \psi\}}\psi_\sf{v}(c_\z),
\end{align}
where $\ell$ is Bruhat length, as in \S\ref{subsec:lusztig}, and $\{-, -\}$ is the exotic Fourier transform, as in \Cref{sec:dl}.
\end{thm} 

\begin{proof}
When $G$ is defined over a finite field and $\bb{k}$ is its algebraic closure, this follows from Lusztig's~\cite[Corollary~14.11]{lusztig_85_3} and~\cite[Theorem~23.1(c)]{lusztig_86_5}.
See also the explanation of~\cite[413]{ww}.

When $\bb{k}$ is of characteristic zero, we deduce the result by a standard ``spreading out'' argument, following~\cite[\S{2.7}]{springer_78} and the references to SGA 4 there.
In brief, we spread out $G$ and the relevant $\QL$-complexes to a model over a finitely generated ring $R \subseteq \bb{k}$ in which $\ell$ is invertible, using the specialization theorem for \'etale cohomology.
After shrinking $R$, we may assume that the complexes are locally constant over $\ur{Spec}(R)$.
At some closed point $\fr{p} \in \ur{Spec}(R)$ and geometric point $\bar{\fr{p}}$ over $\fr{p}$, the multiplicity formula holds at $\bar{\fr{p}}$.
Hence, by the comparison theorem for \'etale cohomology, the same formula holds at the geometric point of $\ur{Spec}(R)$ corresponding to the embedding of $R$ into $\bb{k}$.
\end{proof}

\subsection{The Springer Resolution}\label[subsection]{subsec:setup-springer}

Let $G^\unip$ be the variety of unipotent elements of $G$.
Recall that $\dim G^\unip = 2N$, where $N$ is the number of positive roots of $G$.
Let $\pi_\z^\unip \colon Y_\z^\unip \to G^\unip$ and $\bar{\pi}_\z^\unip \colon \bar{Y}_\z^\unip \to G^\unip$ be defined by the pullback squares:
\begin{equation}
\begin{tikzcd}
Y_\z^\unip \arrow{r} \arrow{d}[left]{\pi_\z^\unip} &Y_\z \arrow{d}{\pi_\z}\\
G^\unip \arrow{r} &G
\end{tikzcd}
\qquad 
\begin{tikzcd}
\bar{Y}_\z^\unip \arrow{r} \arrow{d}[left]{\bar{\pi}_\z^\unip} &\bar{Y}_\z \arrow{d}{\bar{\pi}_\z}\\
G^\unip \arrow{r} &G
\end{tikzcd}
\end{equation}
As mentioned in \S\ref{subsec:lusztig}, $\pi_e^\unip = \bar{\pi}_e^\unip$ is called the \dfemph{Springer resolution} of $G^\unip$.

The \dfemph{Springer sheaf} is $\cal{S} \vcentcolon= \pi_{e, !}^\unip\underline{\QL}_{Y_e^\unip}[2N] = \pi_{e, \ast}^\unip\underline{\QL}_{Y_e^\unip}[2N]$.
Since $Y_e$ is smooth, being smooth over $\cal{B}$, we have $\cal{S}[-2N] = \pi_{e, !}\underline{\QL}_{Y_e}|_{G^\unip} = \bar{K}_e|_{G^\unip}$.
Thus \eqref{eq:grothendieck} restricts to a decomposition
\begin{align}\label{eq:springer}
\cal{S} \simeq \bigoplus_{\chi \in \Irr(W)} V_\chi \otimes \cal{S}_\chi,
\end{align}
where $\cal{S}_\chi = \cal{A}_\chi[2N - \dim G]|_{G^\unip}$.
In particular, the $W$-action on $\bar{K}_e$ restricts to an action on $\cal{S}$, for which we have $W$-equivariant isomorphisms $V_\chi \simeq \Hom(\cal{S}_\chi, \cal{S})$.

In \Cref{sec:unipotent}, we actually need $G$-equivariant versions of the Springer sheaf and its summands.
Let $G$ act on itself by (right) conjugation.
Then $Y_\z$, $\bar{Y}_\z$ are stable under the $G$-action on $\cal{B} \times G$.
Since $G^\unip$ is stable under conjugation, $Y_\z^\unip$, $\bar{Y}_\z^\unip$ are also stable under the $G$-action on $\cal{B} \times G$.
So the maps 
%$\pi_\z$, $\bar{\pi}_\z$ and their restrictions 
$\pi_\z^\unip$, $\bar{\pi}_\z^\unip$ are equivariant.
In particular, $\cal{S}$ and its summands $\cal{S}_\chi$ can be upgraded to objects of the $G$-equivariant derived category of $\QL$-complexes on $G^\unip$~\cite[\S{6.4}]{achar}.
We denote the equivariant version of $\cal{S}$ by $\cal{S}_G$.

In any $G$-equivariant derived category of $\QL$-complexes, we write $\Hom^n(K, K') = \Hom(K, K'[n])$.
The following is an analogue of~\cite[Corollary~6.4]{lusztig_88}, with $G^\unip$ in place of the nilpotent cone.

\begin{thm}[Lusztig]\label[thm]{thm:end-spr}
\begin{enumerate}
\item 	$\Hom^i(\cal{S}_G, \cal{S}_G)$ vanishes in odd degrees $i$.

\item 	There is an isomorphism of graded $\QL$-algebras
\begin{align}\label{eq:end-spr}
\Hom^{2\ast}(\cal{S}_G, \cal{S}_G)
	\simeq \Sym^\ast(\QL \otimes \mathbb{X}) \rtimes \QL W,
\end{align} 
where $\mathbb{X}$ is the character lattice in the root datum of $G$, as in \S\ref{subsec:lusztig}, and $\QL W$ is placed in degree $0$.
The degree-$0$ isomorphism $\Hom^0(\cal{S}_G, \cal{S}_G) \simeq \QL W$ recovers the $W$-action on $\cal{S}_G$ arising from \eqref{eq:springer}.

\end{enumerate}
\end{thm} 

%Let $\bb{H}_W$ denote the graded algebra above.
As in~\cite{trinh}, we may identify $\Hom^\ast(\cal{S}_G, \cal{S}_G)$ with the convolution algebra formed by the equivariant Borel--Moore homology of the Steinberg variety from \eqref{eq:steinberg}.

\section{The Regular Semisimple Locus}\label[section]{sec:rs}

The goal of this section is to prove \Cref{thm:rs}.
We keep the hypotheses and notation of \Cref{sec:setup}, but now, take $\bb{k} = \bb{C}$.

Recall that $\ur{IH}^\ast$ denotes \'etale intersection cohomology with $\QL$-coefficients, and that $\bar{Y}_{\z, g}$ is the fiber of $\bar{Y}_\z$ above $g \in G(\bb{C})$.
For any $\QL$-complex $K$ on $G$ and $g \in G(\bb{C})$, we write $\cal{H}_g^\ast(K)$ for the stalk at $g$ of the cohomology sheaf $\cal{H}^\ast(K)$.

\begin{lem}\label[lem]{lem:ih}
For any $\z \in W$ and $g \in G^\rs(\bb{C})$, we have a $W$-equivariant isomorphism
\begin{align} 
	\ur{IH}^\ast(\bar{Y}_{\z, g}) \simeq \cal{H}_g^\ast(\bar{K}_\z).
\end{align} 
\end{lem} 

\begin{proof}
Let $i_g \colon \bar{Y}_{\z, g} \to \bar{Y}_\z$ be the inclusion map.
%, and let $i_g^\ast$ be ordinary (derived) pullback along $i_g$.
Let $j_g \colon Y_{\z, g} \to \bar{Y}_{\z, g}$ be the inclusion map, parallel to our earlier notation $j \colon Y_\z \to \bar{Y}_\z$.
%, and let $j_{g, !\ast}$, denote intermediate extension along $j_g$.
Then 
\begin{align}
\ur{IH}^\ast(\bar{Y}_{\z, g}) &= \ur{H}^\ast(\bar{Y}_{\z, g}, j_{g, !\ast} i_g^\ast \underline{\bb{C}}_{Y_\z}),\\
\cal{H}_g^\ast(\bar{K}_\z) &= \ur{H}^\ast(\bar{Y}_{\z, g}, i_g^\ast j_{!\ast} \underline{\bb{C}}_{Y_\z}). 
\end{align}
The usual base change theorem gives a $\pi_1(G^\rs, g)$-equivariant isomorphism of (derived) functors $i_g^\ast j_! \xrightarrow{\sim} j_{g, !} i_g^\ast$.
So it suffices to show that the composition $i_g^\ast j_! \xrightarrow{\sim} j_{g, !} i_g^\ast \to j_{g, !\ast} i_g^\ast$ factors through an isomorphism $i_g^\ast j_{!\ast} \xrightarrow{\sim} j_{g, !\ast} i_g^\ast$.

This holds on some Zariski dense open subset of $G^\rs$ by the ``generic base change'' theorem of~\cite[Th\'eor\`eme~1.9]{deligne}.\footnote{See also~\url{https://mathoverflow.net/a/123787}.}
But by~\cite[Proposition~8.2]{an_25}, $\bar{Y}_\z(\bb{C}) \to G(\bb{C})$ restricts to a topological fibration over $G^\rs(\bb{C})$.\footnote{I thank Patrick Brosnan, Jaehyun Hong, Donggun Lee, and Eunjeong Lee for pointing out this reference.}
Therefore, the desired isomorphism holds for all $g \in G^\rs(\bb{C})$.
\end{proof}

The preceding lemma reduces the study of $\ur{IH}^\ast(\bar{Y}_{\z, g})$ to the study of $\cal{H}_g^\ast(\cal{A})$ for unipotent character sheaves $\cal{A}$.

\begin{lem}\label[lem]{lem:monodromy}
For any $\chi \in \Irr(W)$ and $g \in G^\rs(\bb{C})$, the graded vector space $\cal{H}_g^\ast(\cal{A}_\chi)$ is concentrated in degree $-{\dim G}$.
The monodromy action of $\pi_1(G^\rs, g)$ on it defines a representation of $W$ with character $\chi$. 
\end{lem}

\begin{proof}
Let $Y_e^\rs \to G^\rs$ be the pullback of $Y_e \to G$.
Since $Y_e^\rs \to G^\rs$ is a $W$-torsor, and $W$ is finite, we know that $\cal{H}_g^\ast(\bar{K}_e)$ is concentrated in degree $0$.
Since $\cal{A}_\chi$ is a summand of $\bar{K}_e[\dim G]$, we deduce the first claim.

By construction, the monodromy action on the $W$-torsor $\cal{H}_g^0(\bar{K}_e)$ factors through the regular representation of $W$.
The resulting $W$-action commutes with the one arising from the Galois action on $\bar{K}_e$.
So, from comparing \eqref{eq:grothendieck} to the bi-equivariant decomposition $\bb{C} W \simeq \bigoplus_\chi V_\chi \otimes V_{\chi^\vee}$, we deduce that the monodromy action on $\cal{H}_g^\ast(\cal{A}_\chi)$ has character $\chi^\vee$, where $(-)^\vee$ denotes duals.
But irreducible characters of Weyl groups are real-valued, so they are self-dual.
\end{proof} 

\begin{lem}[Lusztig]\label[lem]{lem:vanish}
If $\cal{A}$ is a cuspidal unipotent character sheaf on $G$, then $\cal{H}_g^\ast(\cal{A})$ vanishes for all $g \in G^\rs(\bb{C})$.
\end{lem}

\begin{proof}
Let $Z_G \subseteq G$ be the center and $Z_G^\circ \subseteq Z_G$ its neutral component.

By \cite[(7.1.2)]{lusztig_85_2}, $\cal{A} = \IC(\bar{\Sigma}, \cal{E})[d]$ for some $(G \times Z_G^\circ)$-orbit $\Sigma \subseteq G$, some simple $(G \times Z_G^\circ)$-equivariant local system $\cal{E}$ on $\Sigma$, and some integer $d$.
Moreover, by~\cite[Theorem~23.1(a)]{lusztig_86_5}, $\cal{A}$ is \dfemph{clean}:
Its restriction to $\bar{\Sigma} - \Sigma$ is zero.
(Strictly speaking, \emph{loc.~cit.}\ is for $G$ over the algebraic closure of a finite field.
We deduce the result over $\bb{C}$ by spreading out, as in the proof of \Cref{thm:exotic}.
Note that \emph{loc.~cit.} relies on definitions in~\cite[(13.9.2)]{lusztig_85_3} and~\cite[(7.7)]{lusztig_85_2}.)
It remains to show that $\Sigma$ cannot contain $g$.

If $G$ is a torus, then the only unipotent character sheaf is the constant sheaf, so $G$ cannot be a torus.
So $G_g/Z_G^\circ$ must be nontrivial, where $G_g \subseteq G$ is the centralizer of $g$.
But if $g \in \Sigma$, then~\cite[(7.1.2)]{lusztig_85_2} shows that $G_g/Z_G^\circ$ must be unipotent, whereas $G_g$ is a torus.
\end{proof}

\begin{proof}[Proof of \Cref{thm:rs}]
Fix $g \in G^\rs(\bb{C})$ and $w \in W$.
We are claiming that
\begin{align}\label{eq:tau-proof}
\sf{v}^{\ell(\z)}
	\tau_{\z, \sf{v}}(w) = \sum_i {\sf{v}^{2i}} \tr(w \mid \ur{IH}^{2i}(\bar{Y}_{\z, g})),
\end{align}
where the $W$-action on $\ur{IH}^\ast(\bar{Y}_{\z, g})$ arises from the monodromy action of $\pi_1(G^\rs, g)$.
In the ring of polynomials in $\sf{v}$ over the representation ring of $W$ with $\QL$-coefficients, we have
\begin{align}
&\sum_j {(-\sf{v})^j}\,
	\ur{IH}^j(\bar{Y}_{\z, g})\\
	&= \sum_j {(-\sf{v})^j}\,
			\cal{H}_g^j(\bar{K}_\z)
		&&\text{by \Cref{lem:ih}}\\
	&= \sum_{i, j} {(-\sf{v})^j}\,
			\cal{H}_g^{j - i}({}^p\cal{H}^i(\bar{K}_\z))
		&&\text{by the decomposition theorem}\\
	&= \sum_k {(-\sf{v})^{k}}\, \cal{H}_g^k
			\left(\sum_i {(-\sf{v})^i}\, {}^p\cal{H}^i(\bar{K}_\z)\right)\\
	&= \sum_k {(-\sf{v})^{k}}\, \cal{H}_g^k
			\left(\sum_{i, \chi} {(-\sf{v})^i}(\cal{A}_\chi : {}^p\cal{H}^i(\bar{K}_\z))\cal{A}_\chi\right)
		&&\text{by \Cref{lem:vanish}}\\
	&= \sf{v}^{\ell(\z)} \sum_{k, \chi, \psi} {(-\sf{v})^{k + \dim G}} {\{\chi, \psi\}}\psi_\sf{v}(c_\z)\,
			\cal{H}_g^k(\cal{A}_\chi)
		&&\text{by \Cref{thm:exotic}}.
\end{align}
By \Cref{lem:monodromy}, $\cal{H}_g^\ast(\cal{A}_\chi)$ is concentrated in degree $-{\dim G}$, and the monodromy action on it factors through a representation of $W$ with character $\chi$.
Combining this with the preceding chain of equalities,
\begin{align}
\sf{v}^{\ell(\z)} \sum_{\chi, \psi}
	{\{\chi, \psi\}}\psi_\sf{v}(c_\z)\chi(w)
	= \sum_j {(-\sf{v})^j}\,
	\tr(w \mid \ur{IH}^j(\bar{Y}_{\z, g})).
\end{align}
The left-hand side is $\sf{v}^{\ell(\z)} \tau_{\z, \sf{v}}(w)$.

Finally, Theorem 1.6(1) and display (9.3) in~\cite{an_25} together show that when $g$ is regular semisimple, $\ur{IH}^\ast(\bar{Y}_{\z, g})$ vanishes in odd degrees and $\dim \bar{Y}_{\z, g} = \dim Y_{\z, g} = \ell(\z)$, as mentioned in \S\ref{subsec:lusztig}.
So the right-hand side of the last display simplifies to the right-hand side of \eqref{eq:tau-proof}.
\end{proof}

\begin{rem}\label[rem]{rem:parabolic}
Fix a parabolic subgroup $W' \subseteq W$.
%For any character $\psi \in \Irr(W)$, let $\dim \psi^{W'}$ be the scalar product of $\psi|_{W'}$ with the trivial character of $W'$.
Abreu--Nigro's \cite[Theorem~1.5]{an_25} asserts that for all $\z \in W$ and $g \in G^\rs(\bb{C})$, we have
\begin{align}
\sf{v}^{\ell(\z)} \sum_\psi 
	\psi_\sf{v}(c_\z) \dim V_\psi^{W'}
	= \sum_i \sf{v}^{2i} \dim \ur{IH}^{2i}(\bar{Y}_{\z, g})^{W'}.
\end{align}
By comparing with \Cref{thm:rs}, and using the fact that the character table of the Hecke algebra specializes to that of $W$, which is invertible, we deduce the following curious property of the exotic Fourier transform:
\begin{align}\label{eq:exotic-parabolic}
\sum_\chi 
{\{\chi, \psi\}}\dim V_\chi^{W'}
	= \dim V_\psi^{W'}
	\quad\text{for all $\psi \in \Irr(W)$}.
\end{align}
Compare to the \emph{tower equivalence} property of $\{-, -\}$ discussed in~\cite{michel}.
\end{rem}

\section{The Unipotent Locus}\label[section]{sec:unipotent}

The goal of this section is to prove \Cref{thm:unipotent} as well as its generalization to arbitrary $\z \in W$.
We keep the hypotheses and notation of \Cref{sec:setup}.

We write $K_G$ for the equivariant version of any $G$-stable $\QL$-complex $K$ on a scheme of finite type with a $G$-action.
Recall that $\ur{H}_{c, G}^\ast$ denotes $G$-equivariant compactly-supported \'etale cohomology with $\QL$-coefficients, and that we set
\begin{align}
	\bar{\cal{Z}}_\z
	&= Y_e^\unip \times_G \bar{Y}_\z\\ 
	&\simeq Y_e^\unip \times_{G^\unip} \bar{Y}_\z^\unip.
\end{align}
Recall from \S\ref{subsec:setup-springer} that we write $2N = \dim G^\unip = \dim Y_e^\unip$.

\subsection{Rational Smoothness}

An element $\z \in W$ is \dfemph{rationally smooth} when the intersection cohomology complex of $\bar{O}_\z$, the Zariski closure of $O_\z$ in $\cal{B} \times \cal{B}$, is a constant sheaf.
In this case, the intersection cohomology complex of $\bar{Y}_\z$ is also constant: $J_\z = \underline{\QL}_{\bar{Y}_\z}$.
This, in turn, implies that $\bar{K}_\z = \bar{\pi}_{\z, !} \underline{\QL}_{\bar{Y}_\z} = \bar{\pi}_{\z, \ast} \underline{\QL}_{\bar{Y}_\z}$.

%, where $N$ is the number of positive roots of $G$.

\begin{lem}\label[lem]{lem:hom-spr}
For any $\z \in W$ and integer $i$, we have 
\begin{align}
\ur{H}_{c, G}^i(\bar{\cal{Z}}_\z)^\vee \simeq \Hom^{2N - i}(\bar{\pi}_{\z, !}^\unip \underline{\QL}_{\bar{Y}_\z^\unip, G}, \cal{S}_G),
\end{align}
realizing the graded action of $\Hom^\ast(\cal{S}_G, \cal{S}_G)$ on $\ur{H}_{c, G}^i(\bar{\cal{Z}}_\z)$ via (derived) composition.
In particular, if $\z$ is rationally smooth, then
\begin{align}
\ur{H}_{c, G}^i(\bar{\cal{Z}}_\z)^\vee \simeq \Hom^{2N - i}(\bar{K}_{\z, G}|_{G^\unip}, \cal{S}).
\end{align}
\end{lem} 

\begin{proof}
Let $\bar{\Pi}_\z^\unip \colon \bar{\cal{Z}}_\z \to Y_e^\unip$ be the pullback of $\bar{\pi}_\z^\unip \colon \bar{Y}_\z^\unip \to G^\unip$.
Let $\omega_{Y_e^\unip, G}$ be the $G$-equivariant dualizing complex of $Y_e^\unip$.
We compute
\begin{align}
\ur{H}_c^i(\bar{\cal{Z}}_\z)^\vee 
	&= \ur{H}_c^i(Y_e^\unip, \bar{\Pi}_{\z,!}^\unip \underline{\QL}_{\bar{\cal{Z}}_\z, G})^\vee\\
	&\simeq \Hom^{-i}(\bar{\Pi}_{\z,!}^\unip \underline{\QL}_{\bar{\cal{Z}}_\z, G}, \omega_{Y_e^\unip, G})
		&&\text{by Verdier}\\
	&\simeq \Hom^{4N - i}(\bar{\Pi}_{\z,!}^\unip \underline{\QL}_{\bar{\cal{Z}}_\z, G}, \underline{\QL}_{Y_e^\unip, G})
		&&\text{by smoothness of $Y_e^\unip$}\\
	&\simeq \Hom^{4N - i}(\pi_e^\ast \bar{\pi}_{\z,!}^\unip \underline{\QL}_{\bar{Y}_\z^\unip, G}, \underline{\QL}_{Y_e^\unip, G})
		&&\text{by base change}\\
	&\simeq \Hom^{2N- i}(\bar{\pi}_{\z,!}^\unip \underline{\QL}_{\bar{Y}_\z^\unip, G}, \cal{S}_G),
\end{align}
as claimed.
\end{proof}

The preceding lemma reduces the study of $\ur{H}_{c, G}^\ast(\bar{\cal{Z}}_\z)$ to that of $\Hom^\ast(\cal{A}_G|_{G^\unip}, \cal{S}_G)$ for unipotent character sheaves $\cal{A}$.
Note that since $\bar{\pi}_\z$ is $G$-equivariant, $\bar{K}_\z$ and its shifted perverse summands are indeed $G$-stable, so the equivariant $\QL$-complexes $\cal{A}_G$ are indeed well-defined.

Recall the complexes $\cal{S}_\chi = \cal{A}_\chi[2N - \dim G]|_{G^\unip}$ in \eqref{eq:springer}.
Since the left $W^\op$-action on the right-hand side of \eqref{eq:end-spr} corresponds to the left $W^\op$-action on the left-hand side arising from \eqref{eq:springer}, \Cref{thm:end-spr} implies:

\begin{lem}[Lusztig]\label[lem]{lem:hom-spr-chi}
\begin{enumerate} 
\item 	$\Hom^i(\cal{S}_{\chi, G}, \cal{S}_G)$ vanishes in odd degrees $i$.

\item 	For all $j$, there is a $W$-equivariant isomorphism
\begin{align} 
\Hom^{2j}(\cal{S}_{\chi, G}, \cal{S}_G)
	\simeq \Sym^j(\QL \otimes \mathbb{X}) \otimes V_\chi,
	\end{align}
	where the $W$-action on the left-hand side arises from $\Hom^0(\cal{S}_G, \cal{S}_G) \simeq \QL W$.

\end{enumerate}
\end{lem} 

Now suppose that the unipotent character sheaf $\cal{A}$ is cuspidal.
By reducing to the case where $G$ is almost simple and adjoint, and using the classification of cuspidal character sheaves on such groups in~\cite{lusztig_86_5}, one can check that $\cal{A}|_{G^\unip}$ is always either zero or a shift of a perverse sheaf on $G^\unip$ that is \emph{cuspidal} in the sense of~\cite[Definition~2.3]{rr}:
For details, see~\cite[\S{2}]{aubert}.
Thus, \cite[Theorem~3.5]{rr}, or rather its analogue with $G^\unip$ in place of the nilpotent cone, implies:

\begin{lem}[Lusztig, Rider--Russell]\label[lem]{lem:rr}
If $\cal{A}$ is a cuspidal unipotent character sheaf on $G$, then $\Hom^\ast(\cal{A}_G|_{G^\unip}, \cal{S}_G)$ vanishes.
\end{lem}

\begin{proof}[Proof of \Cref{thm:unipotent}]
Suppose that $\z$ is rationally smooth, and fix $w \in W$.
We are claiming that
\begin{align}
\sum_j {(-\sf{v})^j} \tr(w \mid \ur{H}_{c, G}^j(\bar{\cal{Z}}_\z))
	= (-1)^{\ell(w)} \sf{v}^{2\dim G + \ell(\z)}
		\mathbb{S}_{\sf{v}^2}(w) \tau_{\z, \sf{v}}(w),
\end{align} 
where the $W$-action on $\ur{H}_{c, G}^\ast(\bar{\cal{Z}}_\z)$ arises from the action of $\Hom^0(\cal{S}_G, \cal{S}_G)$.
Since characters of Weyl groups are self-dual, the trace of $w$ on $\ur{H}_{c, G}^j(\bar{\cal{Z}}_\z)$ is also its trace on $\ur{H}_{c, G}^j(\bar{\cal{Z}}_\z)^\vee$.
In the ring of formal power series in $\sf{v}$ over the representation ring of $W$ with $\QL$-coefficients, we have
\begin{align}
&\sum_j {(-\sf{v})^j}\,
	\ur{H}_{c, G}^j(\bar{\cal{Z}}_\z)^\vee\\
&= \sum_j {(-\sf{v})^j}
	\Hom^{2N - j}(\bar{K}_{\z, G}|_{G^\unip}, \cal{S}_G)
	&&\text{by \Cref{lem:hom-spr}}\\
&= \sum_{i, j} {(-\sf{v})^j}
	\Hom^{2N - j + i}({}^p\cal{H}^i(\bar{K}_{\z, G})|_{G^\unip}, \cal{S}_G)
	&&\text{by the decomp.~thm.}\\
&= \sum_{k, i} {(-\sf{v})^{k + i}} 
	\Hom^{2N - k}({}^p\cal{H}^i(\bar{K}_{\z, G})|_{G^\unip}, \cal{S}_G)\\
&= \sum_{k, i, \chi} {(-\sf{v})^{k + i}}
	(\cal{A}_\chi : {}^p\cal{H}^i(\bar{K}_\z))
	\Hom^{4N - \dim G - k}(\cal{S}_{\chi, G}, \cal{S}_G)
	&&\text{by \Cref{lem:rr}}\\
&= \sf{v}^{\ell(\z)} \sum_{k, \chi, \psi} {(-\sf{v})^k} {\{\chi, \psi\}}\psi_\sf{v}(c_\z)
	\Hom^{4N - k}(\cal{S}_{\chi, G}, \cal{S}_G)
	&&\text{by \Cref{thm:exotic}}\\
&= \sf{v}^{4N + \ell(\z)} \sum_{k, \chi, \psi} {(-\sf{v})^{-j}} {\{\chi, \psi\}}\psi_\sf{v}(c_\z)
	\Hom^j(\cal{S}_{\chi, G}, \cal{S}_G)\\
&= \sf{v}^{4N + \ell(\z)}
	\sum_i {\sf{v}^{-2i}} \Sym^i(\QL \otimes \mathbb{X})
	\otimes 
	\sum_{\chi, \psi} {\{\chi, \psi\}}\psi_\sf{v}(c_\z)
V_\chi
	&&\text{by \Cref{lem:hom-spr-chi}}.
\end{align}
The trace of $w$ on the last formal series is 
\begin{align}
\sf{v}^{4N + \ell(\z)}
	\mathbb{S}_{\sf{v}^{-2}}(w)
	\tau_{\z, \sf{v}}(w).
\end{align}
Let $r$ be the rank of $G$.
%, and let $\varepsilon(w) = (-1)^{\ell(w)}$, as in~\Cref{cor:sign}.
To finish the proof, we use the following trick that also appeared in~\cite{trinh}:
\begin{align}
\mathbb{S}_{\sf{q}^{-1}}(w)
	= \frac{1}{\det(1 - \sf{q}^{-1}w \mid \mathbb{X})}
	= \frac{(-1)^{\ell(w)}\sf{q}^r}{\det(1 - \sf{q}w \mid \mathbb{X})}
	= (-1)^{\ell(w)}\sf{q}^r \mathbb{S}_\sf{q}(w).
\end{align}
Since $2N + r = \dim G$, we arrive at the desired identity.
\end{proof}

%\subsection{LLT Polynomials}

\subsection{Virtual Weight Polynomials}

The results of~\cite{trinh} imply a generalization of \Cref{thm:unipotent} at the cost of greater technicalities.
We sketch the proof below.

For any scheme $X$ of finite type over $\bb{C}$, equipped with a $G$-action, we define the \dfemph{$G$-equivariant virtual weight polynomial} of $X$ as the formal alternating sum
\begin{align}
\ur{E}_G(\sf{v}, X)
	= \sum_i {(-1)^i} \sf{v}^j \gr_j^\sf{W} \ur{H}_{c, G}^i(X),
\end{align} 
where $\sf{W}_{\leq \ast}$ denotes the \dfemph{weight filtration} on $\ur{H}_{c, G}^i(X)$ arising from Hodge theory.
If $W$ acts on $\ur{H}_{c, G}^i(X)$ for all $i$, then $\ur{E}_G(X)$ can be upgraded to a polynomial over the representation ring of $W$ with $\QL$-coefficients.
In particular, for $X = \bar{\cal{Z}}_\z$, the weight filtration coincides with the cohomological filtration, giving 
\begin{align}
\ur{E}_G(\bar{\cal{Z}}_\z) = \sum_i {(-\sf{v})^i} \ur{H}_{c, G}^i(\bar{\cal{Z}}_\z).
\end{align}
This follows from the chain of equalities in the proof of \Cref{thm:unipotent}, computing $\ur{H}_{c, G}^i(\bar{\cal{Z}}_\z)^\vee$ in terms of the spaces $\Hom^\ast(\cal{S}_{\chi, G}, \cal{S}_G)$.
With respect to the natural Weil structures on the $\QL$-complexes, these equalities arise from weight-preserving isomorphisms, and as in~\cite{rr}, the weight filtration on $\Hom^\ast(\cal{S}_{\chi, G}, \cal{S}_G)$ coincides with the filtration by $\Hom$-degree.

The function $X \mapsto \ur{E}_G(\sf{v}, X)$ factors through the Grothendieck ring of schemes of finite type with $G$-actions.
In particular, if $U \subseteq X$ is an open $G$-stable subscheme, then $\ur{E}_G(\sf{v}, X) = \ur{E}_G(\sf{v}, U) + \ur{E}_G(\sf{v}, X \setminus U)$.
This leads us to consider the open, $G$-stable subscheme $\cal{Z}_\z \subseteq \bar{\cal{Z}}_\z$ defined by
\begin{align}
\cal{Z}_\z
	&= Y_e^\unip \times_G Y_\z\\ 
	&\simeq Y_e^\unip \times_{G^\unip} Y_\z^\unip.
\end{align}
Observe that $\ur{H}_c^\ast(\cal{Z}_\z)$ receives a graded $\Hom^\ast(\cal{S}_G, \cal{S}_G)$-action, hence a $W$-action, from an analogue of \Cref{lem:hom-spr}.

In~\cite{trinh}, we prove an analogue of \Cref{thm:unipotent} for arbitrary $\z$, with $\cal{Z}_\z$ in place of $\bar{\cal{Z}}_\z$.
Namely, writing $(\delta_w)_w$ for the standard Hecke basis as in \Cref{sec:intro}, we have
\begin{align} 
\tr(w \mid \ur{E}_G(\cal{Z}_\z))
	= (-1)^{\ell(w)} \sf{v}^{2\dim G} \mathbb{S}_{\sf{v}^2}(w)
		\cdot 
		\sum_{\chi, \psi} {\{\chi, \psi\}}\psi_\sf{v}(\delta_\z)\chi(w).
\end{align} 
For all $y \leq \z$ in the Bruhat order on $W$, let $p_{y, \z}(\sf{q}) \in \bb{Z}[\sf{q}]$ denote the associated Kazhdan--Lusztig polynomial, so that
\begin{align}
\sf{v}^{\ell(\z)} c_\z 
	= \sum_{y \leq \z} 
		p_{y, \z}(\sf{v}^2)
		\delta_y.
\end{align}
We conclude:

\begin{thm}\label[thm]{thm:unipotent-kl}
For any $\z \in W$ and $w \in W$, we have
\begin{align}
\sum_{y \leq \z} 
	p_{y, \z}(\sf{v}^2) \tr(w \mid \ur{E}_G(\cal{Z}_y))
	= (-1)^{\ell(w)} \sf{v}^{2\dim G + \ell(\z)}
		\mathbb{S}_{\sf{v}^2}(w) \tau_{\z, \sf{v}}(w),
\end{align}
where the virtual $W$-action on $\ur{E}_G(\cal{Z}_y)$ arises from the $W$-action on $\ur{H}_{c, G}^\ast(\cal{Z}_y)$: \emph{i.e.}, from convolution.
\end{thm} 

When $\z$ is rationally smooth, $p_{y, \z}(\sf{q}) = 1$ for all $y \leq \z$.
So in this case, \Cref{thm:unipotent-kl} specializes to \Cref{thm:unipotent} by the additivity of $\ur{E}_G$.

\section{Induction Theorems}\label[section]{sec:parabolic}

The goal of this section is to prove \Cref{cor:parabolic}, a refinement of \Cref{prop:parabolic}.
We keep the hypotheses and notation of \Cref{sec:setup}.

We identify $W$ with $N_G(T)/T$ for a fixed maximal torus $T \subseteq G$.
Recall that the system of simple reflections $S \subseteq W$ determines and is determined by a choice of Borel $B \subseteq G$ containing $T$.

Fix a subset $S' \subseteq S$.
Let $W'$ be the parabolic subgroup of $W$ generated by $S'$.
Let $H_{W'}$ be the Hecke algebra of $W'$.
The inclusion $S' \subseteq S$ induces an inclusion $H_{W'} \subseteq H_W$.

In what follows, we write $\tau_{\z, \sf{v}}^W$, $\tau_{\z, \sf{v}}^{W'}$ in place of $\tau_{\z, \sf{v}}$, to avoid ambiguity.
Let $\Ind_{W'}^W$ denote induction of class functions from $W'$ to $W$.
The following result is proved in~\cite{trinh}, as part of a more general compatibility for the linear functions on $H_{W'}$ and $H_W$ defined by $c_\z \mapsto \tau_{\z, \sf{v}}^{W'}$ and $c_\z \mapsto \tau_{\z, \sf{v}}^W$.

\begin{thm}\label[thm]{thm:tau-ind}
	For all $\z \in W'$, we have
	\begin{align}
		\tau_{\z, \sf{v}}^W = \Ind_{W'}^W \tau_{\z, \sf{v}}^{W'}
	\end{align}
	as class functions on $W$.
\end{thm}

Let $P = BW'B$, a parabolic subgroup of $G$.
Every parabolic of $G$ arises in this way from some choice of Borel pair $T \subseteq B \subseteq G$ and subset $S' \subseteq S$, since all Borel pairs are conjugate.

Let $G'$ be the Levi factor of $P$, viewed as a connected algebraic subgroup of $G$.
Let $\cal{B}'$ and $(G')^\unip$ be the flag variety and unipotent locus of $G'$, respectively.
For any $u \in (G')^\unip(\bb{k})$, let $\spr_{u, G} \colon W \to \bb{Z}$ and $\spr_{u, G'} \colon W' \to \bb{Z}$ be defined by
\begin{align}
\spr_{u, G}(w) &= \tr(w \mid \ur{H}^\ast(\cal{B}_u))\\
\spr_{u, G'}(w) &= \tr(w \mid \ur{H}^\ast(\cal{B}'_u)).
\end{align}
For $\bb{k}$ of characteristic zero, Alvis--Lusztig observed that
\begin{align}\label{eq:alvis-lusztig}
\Ind_{W'}^W \spr_{u, G'} = \spr_{u, G}
\end{align}
as class functions on $W$.
Lusztig later gave a proof valid for arbitrary $\bb{k}$; for $\bb{k} = \bb{C}$, Treumann gave a proof via equivariant localization and Rossman's refinement of the Springer action to an action on homotopy types.

Below, we need a refinement of \eqref{eq:alvis-lusztig} to the level of the characters $\spr_{\psi, G}$.
%As it is Treumann's proof that we will refine, we assume throughout the rest of the section that $\bb{k} = \bb{C}$.
I am grateful to Roman Bezrukavnikov for suggesting that the proof of \eqref{eq:alvis-lusztig} via equivariant localization admits such a refinement.

Parallel to our notation in \eqref{eq:spr}, let $A'_u = \pi_0(G'_u)$, and for any $\kappa' \in \Irr(A'_u)$, let $\spr_{u, \kappa'} \colon W' \to \bb{Z}$ be defined by
\begin{align}
\spr_{u, \kappa'}(w) = \sum_i \tr(w \mid \ur{H}^{2i}(\cal{B}'_u)_{\kappa'}).
\end{align}
The inclusion $G'_u \subseteq G_u$ descends to a homomorphism $A'_u \to A_u$.
The following result allows us to treat $A'_u$ as a subgroup of $A_u$.
I am grateful to Eric Sommers for help with the proof.

\begin{lem}\label[lem]{lem:inj}
The homomorphism $A'_u \to A_u$ is always injective.
\end{lem} 

\begin{proof}
As explained in \cite[\S{1.8}]{carter}, we can always write $G \simeq (\tilde{H} \times \tilde{T})/Z$, where $\tilde{H}$ is a direct product of almost-simple algebraic groups, $\tilde{T}$ is a torus, and $Z$ is a finite central algebraic subgroup of $\tilde{H} \times \tilde{T}$.
The preimage of $G'$ in $\tilde{H} \times \tilde{T}$ then takes the form $\tilde{H}' \times \tilde{T}$, where $\tilde{H}'$ is a Levi of $\tilde{H}$.
Since $u$ is unipotent, $u$ must belong to the image of $\tilde{H}'(\bb{k})$ in $G'(\bb{k})$.
So it suffices to prove the lemma with $G$ and $G'$ replaced by $\tilde{H}$ and $\tilde{H}'$.
Then, by multiplicativity, it suffices to prove the lemma in the cases where $G$ is almost-simple.

The latter cases can be handled by Sommer's refined Bala--Carter theorem~\cite{sommers}.
In his terminology, we can assume that $u$ is a distinguished unipotent element in some pseudo-Levi subgroup $G''$ of $G'$.
Then $G''$ remains a pseudo-Levi subgroup of $G$.
So the main result of~\cite{sommers} implies that the homomorphism $A'_u \to A_u$ embeds the set of conjugacy classes of $A'_u$ into the set of conjugacy classes of $A_u$.
\end{proof}

\begin{rem} 
Surprisingly, we were unable to find \Cref{lem:inj} stated in the literature.
It would be interesting to have a proof that avoids reduction to the case where $G$ is almost-simple.
\end{rem} 

\begin{thm}\label[thm]{thm:ind}
Suppose that $\bb{k} = \bb{C}$.
Then for any $u \in (G')^\unip(\bb{C})$ and $\kappa' \in \Irr(A'_u)$, we have
\begin{align}
\Ind_{W'}^W \spr_{u, \kappa'}
=
\sum_{\kappa \in \Irr(A_u)}
	(\kappa, \Ind_{A'_u}^{A_u} \kappa')_{A_u} \spr_{u, \kappa},
\end{align}
where $(-, -)_{A_u}$ is the usual scalar product of class functions.
\end{thm}

\begin{ex} 
Recall that the regular representation of a finite group, like $A'_u$, remains a regular representation when induced to a larger finite group, like $A_u$.
From this fact and the linearity of induction, we see that \Cref{thm:ind} implies \eqref{eq:alvis-lusztig}.
\end{ex} 

\begin{ex} 
Take $G$ and $G'$ of types $B_3$ and $B_2$, respectively.
The embedding of $G'$ into $G$ is unique up to conjugacy.
The unipotent classes of $G$, \emph{resp.}\ $G'$, are labeled by certain integer partitions of $7$, \emph{resp.} $5$ in \cite{carter} or CHEVIE.
The irreducible characters of $W$, \emph{resp.}\ $W'$, are labeled by the integer bipartitions of $3$, \emph{resp.}\ $2$.

Let $u$ belong to the unipotent conjugacy class in $G'$ labeled $311$.
Its conjugacy class in $G$ is labeled $31111$.
The groups $A_u$, $A'_u$ are both cyclic of order $2$.
Writing $\varepsilon$, $\varepsilon'$ for their respective nontrivial characters, we have the following decompositions of the characters $\spr_{u, \kappa}$, $\spr_{u, \kappa'}$:
\begin{align}
\tiny\begin{array}[t]{r|ll}
	&\spr_{u, 1_{A_u}}
	&\spr_{u, \varepsilon}\\
	\hline 
3.
	&1
	&\\
21.
	&1
	&1\\
2.1
	&2
	&\\
.3
	&
	&\\
1.2
	&1
	&\\
11.1
	&1
	&1\\
111.
	&
	&1\\
1.11
	&1
	&\\
.21
	&
	&\\
.111
	&
	&
\end{array}
\qquad 
\begin{array}[t]{r|ll}
	&\spr_{u, 1_{A'_u}}
	&\spr_{u, \varepsilon'}\\
	\hline 
2.
	&1
	&\\
11.
	&
	&1\\
1.1
	&1
	&\\
.2
	&
	&\\
.11
	&
	&\\
\end{array}
\end{align}
The map $A'_u \to A_u$ is an isomorphism.
So here, \Cref{thm:ind} predicts that \begin{align} 
\Ind_{W'}^W \spr_{u, 1_{A'_u}} 
	= \spr_{u, 1_{A_u}}
\quad\text{and}\quad 
\Ind_{W'}^W \spr_{u, \varepsilon'} 
	= \spr_{u, \varepsilon}.
\end{align}
Indeed, the above identities are in agreement with the following table of the scalar products $(\chi, \Ind_{W'}^W \chi')_W$ for $\chi \in \Irr(W)$ and $\chi' \in \Irr(W')$:
\begin{align}
\tiny\begin{array}[t]{r|lll}
	&2.
	&11.
	&1.1\\
		\hline 
3.
	&1
	&
	&\\
21.
	&1
	&1
	&\\
2.1
	&1
	&
	&1\\
.3
	&
	&
	&\\
1.2
	&
	&
	&1\\
11.1
	&
	&1
	&1\\
111.
	&
	&1
	&\\
1.11
	&
	&
	&1\\
.21
	&
	&
	&\\
.111
	&
	&
	&
\end{array}
\end{align}
\end{ex} 

\begin{proof}[Proof of \Cref{thm:ind}]
We will merely sketch what adjustments are needed to the proof of Theorem 1.1 in \cite[15--16]{treumann}.

Note that Treumann's $L$, $\nu$, $(-)^\ur{rss}$ are our $G'$, $u$, $(-)^\rs$.
In particular, his $W_G$, $W_L$ are our $W$, $W'$.
Also, he works with the topological spaces formed by the $\bb{C}$-points of his algebraic varieties, in the analytic topology, rather than the varieties themselves in the \'etale topology.
Via the comparison theorem of~\cite[Expos\'e~XVI]{agvdd}, his conclusions for singular cohomology can be transported to \'etale cohomology, once we extend scalars from $\bb{Q}$ to $\QL$.

As in~\cite{treumann}, let $\ur{U}(1)$ be the compact real form of $\GL_1$, and fix an embedding of $\ur{U}(1)$ into $G$ such that the neutral component of its centralizer is precisely $G'$.
The main structure in the proof is a commutative square of topological spaces, which in our notation becomes:
\begin{equation}
\begin{tikzcd}
E_{u, G'}^\rs \arrow{r} \arrow{d}[left]{\chi'}
	&E_{u, G}^\rs \arrow{d}{\chi}\\
B(\delta, T \sslash W')^\rs \arrow{r}{\iota} 
	&B(\delta, T \sslash W)^\rs
\end{tikzcd}
\end{equation}
Above:
The notation $B(\delta, -)$ denotes an open ball of radius $\delta > 0$ around a chosen point.
The bottom horizontal arrow is a fiber bundle with fiber $W/W'$.
The vertical arrows are Ehresmann fibrations: that is, smooth, proper, surjective maps.
The fibers of $\chi'$, \emph{resp.}\ $\chi$, are homotopy equivalent to $\cal{B}_u$, \emph{resp.}\ $\cal{B}'_u$.
There is a $\ur{U}(1)$-action on $E_{u, G}^\rs$, along which $\chi$ is invariant.
The top horizontal arrow is an embedding that identifies $E_{u, G'}^\rs$ with the set of $\ur{U}(1)$-fixed points of $E_{u, G}$.
%The notation $(-)_{h\ur{U}(1)}$ denotes the Borel construction applied to a space with a $\ur{U}(1)$-action.

Let $\underline{\bb{Q}}_{(-), \ur{U}(1)}$ denote the $\ur{U}(1)$-equivariant constant sheaf with $\bb{Q}$-coefficients over a $\ur{U}(1)$-space.
The monodromy action of the fundamental group of $B(\delta, T \sslash W)^\rs$, \emph{resp.}\ $B(\delta, T \sslash W')^\rs$, on the stalks of $\underline{\bb{Q}}_{E_{u, G}^\rs, \ur{U}(1)}$, \emph{resp.}\ $\underline{\bb{Q}}_{E_{u, G'}^\rs, \ur{U}(1)}$, factors through $W$, \emph{resp.}\ $W'$.
This induces the Springer action on $\ur{H}_{\ur{U}(1)}^\ast(\cal{B}_u, \bb{Q})$, \emph{resp.}\ $\ur{H}_{\ur{U}(1)}^\ast(\cal{B}'_u, \bb{Q})$, with the same sign convention as in \Cref{sec:unipotent}.

As in~\cite{treumann}, we identify $\ur{H}_{\ur{U}(1)}^\ast(\point, \bb{Q})$ with $\bb{Q}[t]$, where $\deg t = 2$.
The main input in Treumann's proof is the localization theorem for $\ur{U}(1)$-equivariant cohomology, showing that in the $\ur{U}(1)$-equivariant derived category of sheaves in the analytic topology on $B(\delta, T \sslash W)^\rs$, we have a $W$-equivariant isomorphism:
\begin{align}\label{eq:equivariant-loc}
\iota_\ast \chi'_\ast \underline{\bb{Q}}_{E_{u, G'}^\rs, \ur{U}(1)} \otimes_{\bb{Q}[t]} \bb{Q}[t^{\pm 1}]	&\simeq 
	\chi_\ast \underline{\bb{Q}}_{E_{u, G}^\rs, \ur{U}(1)} \otimes_{\bb{Q}[t]} \bb{Q}[t^{\pm 1}].
\end{align}
This induces a $W$-equivariant isomorphism:
\begin{align}\label{eq:equivariant-loc-cohomology}
\Ind_{W'}^W \ur{H}_{\ur{U}(1)}^\ast(\cal{B}'_u, \bb{Q}) \otimes_{\bb{Q}[t]} \bb{Q}[t^{\pm 1}]	
	&\simeq 
	\ur{H}_{\ur{U}(1)}^\ast(\cal{B}_u, \bb{Q}) \otimes_{\bb{Q}[t]} \bb{Q}[t^{\pm 1}].
\end{align}
The remainder of Treumann's proof amounts to showing the degeneration of a spectral sequence, which in turn shows how ordinary cohomology can be recovered from equivariant cohomology after collapsing cohomological degree.

Since the copy of $\ur{U}(1)$ in $G$ commutes with $G'$ by construction, we see that the $G'_u$-actions on $\cal{B}_u$ and $\cal{B}'_u$ induce $A'_u$-actions on $\ur{H}_{\ur{U}(1)}^\ast(\cal{B}_u, \bb{Q})$ and $\ur{H}_{\ur{U}(1)}^\ast(\cal{B}'_u, \bb{Q})$.
These in turn induce $A'_u$-actions on $\underline{\bb{Q}}_{E_{u, G}^\rs, \ur{U}(1)}$ and $\underline{\bb{Q}}_{E_{u, G'}^\rs, \ur{U}(1)}$.
Since $E_{u, G'}^\rs$ is just the $\ur{U}(1)$-fixed locus of $E_{u, G}$, the isomorphism \eqref{eq:equivariant-loc} is equivariant with respect to these actions.
So \eqref{eq:equivariant-loc-cohomology} is also $A'_u$-equivariant.
So for any $\kappa' \in \Irr(A'_u)$, we have
\begin{align}
\Ind_{W'}^W \ur{H}_{\ur{U}(1)}^\ast(\cal{B}'_u, \bb{Q})_{\kappa'} \otimes_{\bb{Q}[t]} \bb{Q}[t^{\pm 1}]	
	&\simeq 
	\ur{H}_{\ur{U}(1)}^\ast(\cal{B}_u, \bb{Q})_{\kappa'} \otimes_{\bb{Q}[t]} \bb{Q}[t^{\pm 1}],
\end{align}
where, on the right-hand side, we implicitly pull back $\ur{H}_{\ur{U}(1)}^\ast(\cal{B}_u, \bb{Q})$ to a representation of $A'_u$.
The spectral-sequence degeneration is compatible with this refinement to $\kappa'$-multiplicity spaces.
Altogether, 
\begin{align}
\Ind_{W'}^W \spr_{u, \kappa'}
	=
	\sum_{\kappa \in \Irr(A_u)}
	(\Res_{A'_u}^{A_u} \kappa, \kappa')_{A'_u} \spr_{u, \kappa},
\end{align}
which is equivalent to the theorem statement by Frobenius reciprocity.
\end{proof}

In what follows, let $\Upsilon_G$ be the set of $G$-conjugacy classes of pairs $(u, \kappa)$ such that $u \in G^\unip(\bb{k})$ and $\kappa \in \Irr(A_u)$.
Let $\Upsilon_{G'}$ be defined similarly, with $G$ replaced by $G'$.
For fixed $u \in G^\unip(\bb{k})$, let $\Upsilon_{G', u} \subseteq \Upsilon_{G'}$ be the subset of classes $[u', \kappa']$ such that $u'$ becomes conjugate to $u$ in $G$.

When $[u, \kappa] \in \Upsilon_G$ is the image of $\psi \in \Irr(W)$ under the Springer correspondence, we set $\alpha_{u, \kappa} = \alpha_{\psi, G}$.
For $[u', \kappa'] \in \Upsilon_{G'}$, we define $\alpha_{u', \kappa'}$ similarly.
With this notation, we can state a refinement of \Cref{prop:parabolic}.

\begin{cor}\label[cor]{cor:parabolic}
Suppose that $\bb{k} = \bb{C}$.
%As above, let $W' \subseteq W$ be a parabolic subgroup and $G' \subseteq G$ the corresponding Levi relative to a fixed Borel pair of $G$.
If $\z \in W'$ in the setup above, then
\begin{align}
\alpha_{u, \kappa}^\z 
	= \sum_{[u', \kappa'] \in \Upsilon_{G', u}} 
		(\kappa, \Ind_{A'_{u'}}^{A_u} \kappa')_{A_u}
		\alpha_{u', \kappa'}^\z
		\quad\text{for all $[u, \kappa] \in \Upsilon_G$}.
\end{align}
In particular, the $\alpha_{\psi, G}^\z$ for $\psi \in \Irr(W)$ are nonnegative integer sums of the $\alpha_{\psi', G'}^\z$ for $\psi' \in \Irr(W')$.
\end{cor} 

\begin{proof}
We have
\begin{align}
\tau_{\z, \sf{v}}^W
	&= \Ind_{W'}^W \tau_{\z, \sf{v}}^{W'}
		&&\text{by \Cref{thm:tau-ind}}\\
	&= \sum_{[u', \kappa'] \in \Upsilon_{G'}} 
		\alpha_{u', \kappa'}^\z \Ind_{W'}^W \spr_{u', \kappa'}\\
	&= \sum_{[u', \kappa'] \in \Upsilon_{G'}} 
		\alpha_{u', \kappa'}^\z \sum_{\kappa \in \Irr(A_{u'})} (\kappa, \Ind_{A'_{u'}}^{A_u} \kappa')_{A_u} \spr_{u, \kappa} 
		&&\text{by \Cref{thm:ind}}\\
	&= \sum_{[u] \in G^\unip/G} \sum_{[u', \kappa'] \in \Upsilon_{G', u}} \alpha_{u', \kappa'}^\z \sum_{\kappa \in \Irr(A_u)} (\kappa, \Ind_{A'_{u'}}^{A_u} \kappa')_{A_u} \spr_{u, \kappa}\\
	&= \sum_{[u, \kappa] \in \Upsilon_G} \left( 
		\sum_{[u', \kappa'] \in \Upsilon_{G', u}} 
		(\kappa, \Ind_{A'_{u'}}^{A_u} \kappa')_{A_u}
		\alpha_{u', \kappa'}^\z
		\right) \spr_{u, \kappa}.
\end{align}
Since the decomposition of $\tau_{\z, \sf{v}}^W$ into the characters $\spr_{u, \kappa}$ is unique, we deduce the desired claim.
\end{proof}

\section{The Cell Order}\label[section]{sec:cells}

In this section, we first review two-sided cells following \cite{kl} and \cite[Chapter~5]{lusztig_84}, then prove \Cref{prop:cells}, \Cref{cor:cells}, and some other facts.

Recall that $H_W$ denotes the Hecke algebra of $W$.
For all $x, y \in W$, write $y \xrightarrow{L} x$, \emph{resp.}\ $y \xrightarrow{R} x$, to mean that there exists some $\alpha \in H_W$ such that $c_y$ has nonzero coefficient when $\alpha c_x$, \emph{resp.}\ $c_x \alpha$, is expanded in the Kazhdan--Lusztig basis.
Write $y \xrightarrow{LR} x$ to mean that either $y \xrightarrow{L} x$ or $y \xrightarrow{R} x$.
The relation $\xrightarrow{LR}$ is reflexive, so its transitive closure $\leq_{LR}$ is a partial order on $W$.
The identity $e$ is the unique topmost element in $\leq_{LR}$, while the longest element $w_\circ$ is the bottommost.

Write $x \sim_{LR} y$ to mean that both $y \leq_{LR} x$ and $x \leq_{LR} y$.
Then $\sim_{LR}$ is an equivalence relation on $W$.
The \dfemph{two-sided cells} of $W$ are the equivalence classes under $\sim_{LR}$.
The partial order $\leq_{LR}$ descends to a partial order $\leq$ on the set of two-sided cells $\sf{C}_W$.
The singleton $\{e\}$ is the unique topmost two-sided cell, while the singleton $\{w_\circ\}$ is the bottommost.

For any two-sided cell $\cal{F} \in \sf{C}_W$, let $M_{\leq \cal{F}}$, \emph{resp.}\ $M_{< \cal{F}}$, be the $\bb{Z}[\sf{v}^{\pm 1}]$-submodule of $H_W$ spanned by the elements $c_w$ with $w \in \cal{F}'$ for some $\cal{F}' \leq \cal{F}$, \emph{resp.}\ $\cal{F}' < \cal{F}$.
The definition of $\leq_{LR}$ shows that $M_{\leq \cal{F}}$ and $M_{< \cal{F}}$ are, in fact, two-sided ideals of $H_W$.
We define the \dfemph{two-sided cell module} associated with $\cal{F}$ to be the left $H_W$-module formed by $M_\cal{F} \vcentcolon= M_{\leq \cal{F}}/M_{< \cal{F}}$.

Each simple module over $\bb{Q}(\sf{v}) \otimes H_W$ occurs in the two-sided cell module $M_\cal{F}$ for a unique $\cal{F} \in \sf{C}_W$.
As any such module has a character of the form $\chi_\sf{v}$ for some $\chi \in \Irr(W)$, we obtain a map $\Irr(W) \to \sf{C}_W$, which we will denote by $\chi \mapsto \cal{F}(\chi)$.
By work of Barbasch--Vogan~\cite[\S{5.15}]{lusztig_84}, the pairing $\{-, -\}$ reviewed in \Cref{sec:dl} is block-diagonal with respect to the fibers of this map.
That is:

\begin{lem}[Barbasch--Vogan]\label[lem]{lem:fourier-blocks}
If $\cal{F}(\chi) \neq \cal{F}(\psi)$, then $\{\chi, \psi\} = 0$.
\end{lem} 

The sets $\Irr(W)_\cal{F} = \{\chi \in \Irr(W) \mid \cal{F}(\chi) = \cal{F}\}$ may be called the \dfemph{families} of the irreducible characters of $W$.
Lusztig observed in~\cite{lusztig_97} that under the Springer correspondence, the partition of $\Irr(W)$ into families $\Irr(W)_\cal{F}$
% \vcentcolon= \{\chi \in \cal{F} \mid \cal{F}(\chi) = \cal{F}\}$
corresponds to a partition of $G^\unip$ into subvarieties $G_\cal{F}^\unip$, now called the \dfemph{pieces} of $G^\unip$.
In particular, if $\prec_\unip$, \emph{resp.}\ $\prec_{LR}$, denotes the partial order on $\Irr(W)$ obtained by pulling back the reverse of the closure order on the unipotent conjugacy classes of $G$, \emph{resp.}\ pulling back the partial order on two-sided cells of $W$, then:

\begin{lem}[Lusztig]\label[lem]{lem:partial-order}
$\prec_\unip$ is a refinement of $\prec_{LR}$.
\end{lem} 

\Cref{prop:cells} is part (2) below.

\begin{prop} 
Let $\cal{F} \in \sf{C}_W$ and $\z \in \cal{F}$.
\begin{enumerate} 
\item 	If $\chi \in \Irr(W)$ and $\cal{F} < \cal{F}(\chi)$, then
\begin{align}
\sum_{\chi, \psi} {\{\chi, \psi\}}\psi_\sf{v}(c_\z) = 0.
\end{align}

\item 	If $\psi \in \Irr(W)$ and $\cal{F} < \cal{F}(\psi)$, then $\alpha_{\psi, G}^\z = 0$.

\end{enumerate}
\end{prop} 

\begin{proof}
\begin{enumerate} 
\item 	Fix $\psi \in \Irr(W)$.
If $\{\chi, \psi\} \neq 0$, then $\cal{F}(\psi) = \cal{F}(\chi)$ by \Cref{lem:fourier-blocks}, from which $\cal{F} < \cal{F}(\psi)$, which in turn forces $\psi_\sf{v}(c_\z) = 0$ by the definition of $\leq_{LR}$.
So every term vanishes in the double sum above.

\item 	Observe that the matrix of scalar products $((\chi, \spr_\psi)_W)_{\chi, \psi}$ is unipotent upper-triangular in a total order on $\Irr(W)$ refining $\prec_\unip$.
		Hence its inverse is also upper-triangular in the same total order.
		Now, the claim follows from combining (1) and \Cref{lem:partial-order}.\qedhere

\end{enumerate}
\end{proof}

Let $1_W$ be the trivial character of $W$; let $\varepsilon$ be the sign character defined by $\varepsilon(w) = (-1)^{\ell(w)}$.
\Cref{cor:cells-bottom} is parts (2)--(3) below.

\begin{cor}\label[cor]{cor:cells}
\begin{enumerate}
\item 	If $\z \neq e$, then $\alpha_{\varepsilon, G}^\z = 0$.

\item 	If $\psi \neq 1_W$, then $\alpha_{\psi, G}^{w_\circ} = 0$.

\item 	We have 
		\begin{align}
		\alpha_{1_W, G}^{w_\circ}(\sf{v}) = \displaystyle\sum_{w \in W} \sf{v}^{2\ell(w) - \ell(w_\circ)}.
		\end{align}
	
\end{enumerate}
\end{cor}
	
\begin{proof}
\begin{enumerate}
\item 	This follows from observing that the module for the topmost two-sided cell $\{e\}$ is simple with character $\varepsilon_\sf{v}$: in other words, that $\Irr(W)_{\{e\}} = \{\varepsilon\}$.
		Indeed, $\varepsilon_\sf{v}$ is determined by the property that $\varepsilon_\sf{v}(\delta_s) = -1$ for all simple reflections $s$.
		This means $\varepsilon_\sf{v}(c_s) = 0$ for all $s$, which in turn describes the character of the cell module.

\item 	This follows from observing that the module for the bottommost two-sided cell $\{w_\circ\}$ is simple with character $1_{W, \sf{v}}$: in other words, that $\Irr(W)_{\{w_\circ\}} = \{1_W\}$.
		This in turn follows from (1) and \cite[Lemma~5.14(iii)]{lusztig_84}.

\item 	By (2), we must have $\tau_{w_\circ, \sf{v}} = \alpha_{1_W, G}^{w_\circ} \spr_{1_W}$.
		But $\spr_{1_W} = 1_W$, so we deduce that $\alpha_{1_W, G}^{w_\circ} = (1_W, \tau_{w_\circ, \sf{v}})_W = 1_{W, \sf{v}}(c_{w_\circ})$.
		Since $\sf{v}^{\ell(w_\circ)} c_{w_\circ} = \sum_{w \in W} \delta_w$ and $1_{W, \sf{v}}(\delta_w) = \sf{v}^{2\ell(w)}$, we win.\qedhere

\end{enumerate}
\end{proof}

\section{Examples}\label[section]{sec:examples}

%Observe, as well, that the functions $\tau_{\z, \sf{v}}$ and $\spr_{\psi, G}$ are multiplicative with respect to direct-product factorizations of $W$ that arise from direct-product factorizations of $G$.

Recall $(\vartriangle)_{\psi, G}^\z, (\pm)_{\psi, G}^\z, (+)_{\psi, G}^\z$ from \S\ref{subsec:springer}.
In this section, we summarize their behavior in each irreducible root system outside type $A$ of rank $\leq 5$, as well as in those of rank $6$ when $\z$ is rationally smooth.

Recall that in~\cite{haiman}, Haiman verified \Cref{conj:haiman}---stating that $(\vartriangle)_{\psi, G}^\z$ and $(+)_{\psi, G}^\z$ hold in type $A$---up through rank $6$.

% \Cref{conj:main} for all Weyl groups of rank $\leq 5$---it remains to verify it for the irreducible Weyl groups of types $BC_2$, $G_2$, $BC_3$, $BC_4$, $D_4$, $F_4$, $BC_5$, $D_5$.

\subsection{Methodology}

We can rewrite identity \eqref{eq:main} as the collection of identities
\begin{align}\label{eq:main-chi}
(\chi, \tau_{\z, \sf{v}})_W
	= \sum_{\psi \in \Irr(W)}
		\alpha_{\psi, G}^\z(\sf{v})
		(\chi, \spr_{\psi, G})_W
		\quad\text{for $\chi \in \Irr(W)$}.
\end{align}
Explicitly, the left-hand side is $\sum_\psi {\{\chi, \psi\}}\psi_\sf{v}(c_\z)$.
In the CHEVIE package for GAP3:
\begin{itemize} 
\item 	The pairings $\{\chi, \psi\}$ can be computed from the functions \texttt{UnipotentCharacter} and \texttt{AlmostCharacter}.
\item 	The values $\psi_\sf{v}(c_\z)$ can be computed from \texttt{HeckeCharValues}.
\item 	The pairings $(\chi, \spr_{\psi, G})_W$ can be computed from \texttt{ICCTable}.
		In the GAP3 manual and other literature on the Lusztig--Shoji algorithm, $(\chi, \spr_{\psi, G})_W$ is denoted $P_{\psi, \chi}$:
		See~\cite{gm,achar,achar_rims}.
\end{itemize}
Note that the ordering of $\Irr(W)$ used in \texttt{ICCTable} differs from the ordering used in the other functions above.
After matching the orderings, we obtain the $\alpha_{\psi, G}^\z$ by multiplying the matrix $((\chi, \tau_{\z, \sf{v}})_W)_{\chi, \z}$ on the left by the inverse of the matrix $((\chi, \spr_{\psi, G})_W)_{\chi, \psi}$.
To process the CHEVIE data, we used Python scripts written with help from Claude Sonnet 4.6.
The data and scripts are available at:
\begin{align} \text{\url{https://mqtrinh.github.io/math/research/code/chevie/}}
\end{align} 
The webpage also includes data for the root systems of types $B_7$, $C_7$, $D_7$ when $\z$ is rationally smooth of length $\leq 24$.

%In particular, it helps us to draw upon another observation from~\cite{haiman}.
%For any elements $\beta, \beta' \in H_W$, set \dfemph{$\beta \sim \beta'$} if and only if $\psi_\sf{v}(\beta) = \psi_\sf{v}(\beta')$ for all $\psi \in \Irr(W)$.
%For any subset $\Omega \subseteq W$, say that an element $\z \in W$ is \dfemph{$\Omega$-positive} when $c_\z \sim \sum_{w \in \Omega} m_\z^w c_w$ for some $m_\z^w \in \bb{Z}_{\geq 0}$.
%We see that if $\z \in W$ is $\Omega$-positive, then to verify \Cref{conj:main} for the fixed element $\z$, it suffices to verify it with $w$ in place of $\z$ for all $w \in \Omega$.

\subsection{Notation}

We use the labels in CHEVIE, in the ordering of \texttt{ICCTable},
%in turn based on~\cite{carter},
for irreducible characters of $W$.
We use reduced words in the simple reflections $s_i \in S$ to label elements of $W$, and use the numbering in CHEVIE to label the $s_i$.\footnote{See \url{https://webusers.imj-prg.fr/~jean.michel/gap3/htm/chap085.htm}.}
We write $w_\circ$ for the longest element.

In rank $\leq 3$, we provide tables of the $\alpha_{\psi, G}^\z$, listing $\psi \in \Irr(W)$ along rows and $\z \in W$ along columns.
To compactify the entries, we write
\begin{align} 
	(\sf{a}_m\cdots \sf{a}_1\sf{a}_0) \vcentcolon= \sf{a}_m\sf{v}^m + \cdots + \sf{a}_1\sf{v} + \sf{a}_0 + \sf{a}_1\sf{v}^{-1} + \cdots + \sf{a}_m\sf{v}^{-m}.
\end{align}
%\quad\text{for any positive integers $\sf{a}_0, \sf{a}_1, \ldots, \sf{a}_m$}.
In all ranks, we provide tables stating for each $\psi$ whether $(\vartriangle)_{\psi, G}^\z, (\pm)_{\psi, G}^\z, (+)_{\psi, G}^\z$ fail for some $\z$ and whether $\psi$ is inflated from certain quotients.
Here, we also give the label in CHEVIE for the image of $\psi$ under the Springer correspondence.

For any $w = s_{i_1} \cdots s_{i_\ell} \in W$, we write $\varphi_{i_1\cdots i_\ell}$ to denote the quotient map out of $W$ that is reduction modulo $w$.
We set $\varphi_\circ = \varphi_{w_\circ}$.

\subsection{Type $B_2$}

\begin{comment} 
The table of values $(\chi, \spr_{\psi, G})_W$:
\begin{align}
\begin{array}{r|rrrrr}
&5&311^{(11)}&311&221&\varepsilon\\ 
\hline
{2.}&1&&1&1&1\\
{11.}&&1&&&1\\
{1.1}&&&1&1&2\\
{.2}&&&&1&1\\
{.11}&&&&&1\\
\end{array}
\end{align}
\end{comment}
\begin{comment}
The table of $(\chi, \tau_{\z, \sf{v}})_W$:
\begin{align}
\begin{array}{r|rrrrrrr}
		&e	&t		&s		&ts,st	&tst		&sts	&w_\circ\\
		\hline	
{2.}	&1	&(10)	&(10)	&(102)	&(1020)		&(1020)	&(10202)\\
{11.}	&1	&		&(10)	&1		&(10)		&		&\\
1.1	&2	&(10)	&(10)	&1		&			&		&\\
{.2}	&1	&(10)	&		&1		&			&(10)	&\\
{.11}	&1	&		&		&		&			&		&\\
\end{array}
\end{align}
\end{comment} 
%Note that $s_1$ corresponds to the short simple root.
The $\alpha_{\psi, G}^\z$:
\begin{align}
\tiny\begin{array}{r|rrrrrrr}
	&e	&2		&1		&21, 12		&212		&121	&w_\circ\\
\hline	
2.
	&	&		&		&(101)	&(1020)		&(1020)	&(10202)\\
11.
	&	&		&(10)	&1		&(10)		&		&\\
1.1
	&	&		&(10)	&		&			&-(10)	&\\
.2
	&	&(10)	&		&1		&			&(10)	&\\
.11
	&1	&		&		&		&			&		&\\
\end{array}
\end{align}

The unique maximal quotient map to a product of symmetric groups is $\varphi_\circ \colon W \to W(A_1 \times A_1)$.

{\tiny\begin{longtable}{rl|rrrr}
	&
	&$\neg (\vartriangle)$
	&$\neg (\pm)$
	&$\neg (+)$
	&$\varphi_\circ$-inflation\\
	\hline	
2.	&$5$
	&&&&$\checkmark$\\
11.	&$311(11)$
	&&&&$\checkmark$\\
1.1	&$311$
	&&&$\checkmark$\\
.2	&$221$
	&&&&$\checkmark$\\
.11	&$11111$
	&&&&$\checkmark$\\
\end{longtable}}

\subsection{Type $G_2$}

\begin{comment} 
The table of $(\chi, \spr_{\psi, G})_W$:
\begin{align}
\begin{array}{r|rrrrrr}
&G_2&G_2(a_1)^{(21)}&G_2(a_1)&\tilde A_1&A_1&\varepsilon\\ 
\hline
{\phi_{1,0}}&1&&1&1&1&1\\
{\phi_{1,3}'}&&1&&1&&1\\
{\phi_{2,1}}&&&1&1&1&2\\
{\phi_{2,2}}&&&&1&1&2\\
{\phi_{1,3}''}&&&&&1&1\\
{\phi_{1,6}}&&&&&&1\\
\end{array}
\end{align}
\end{comment} 
\begin{comment} 
The table of $(\chi, \tau_{\z, \sf{v}})_W$:
\begin{align}
\scriptsize\begin{array}{r|rrrrrrrrrr}
				&e	&t		&s		&ts, st	&tst	&sts	&tsts, stst	&tstst		&ststs		&w_\circ\\
				\hline 
\phi_{1,0}		&1	&(10)	&(10)	&(102)	&(1020)	&(1020)	&(10202)	&(102020)	&(102020)	&(1020202)\\
\phi_{1,3}'		&1	&(10)	&		&1		&		&(10)	&1			&(10)		&			&\\
\phi_{2,1}		&2	&(10)	&(10)	&1		&		&(10)	&			&			&			&\\
\phi_{2,2}		&2	&(10)	&(10)	&2		&(10)	&		&1			&			&			&\\
\phi_{1,3}''	&1	&		&(10)	&1		&(10)	&		&1			&			&(10)		&\\
\phi_{1,6}		&1	&		&		&		&		&		&			&			&			&\\
\end{array}
\end{align}
\end{comment}
%Note that $s_1$ corresponds to the long simple root.
The $\alpha_{\psi, G}^\z$:
\begin{align}
\tiny\begin{array}{r|rrrrrrrrrr}
				&e	&2		&1		&21, 12	&212	&121	&2121, 1212	&21212	&12121	&w_\circ\\
				\hline 
\phi_{1,0}		&	&		&		&(101)	&(1020)	&(1020)	&(10202)	&(102020)	&(102020)	&(1020202)\\
\phi_{1,3}'		&	&		&		&		&		&		&1			&(10)		&(10)		&\\
\phi_{2,1}		&	&		&		&-1		&-(10)	&-(10)	&-1			&			&			&\\
\phi_{2,2}		&	&(10)	&		&1		&		&(10)	&			&			&-(10)		&\\
\phi_{1,3}''	&	&		&(10)	&1		&(10)	&		&1			&			&(10)		&\\
\phi_{1,6}		&1	&		&		&		&		&		&			&			&			&\\
\end{array}
\end{align}

The maximal quotient maps to products of symmetric groups are $\varphi_{1212} \colon W \to W(A_1 \times A_1)$ and $\varphi_\circ \colon W \to W(A_2)$.
The latter is what we called the $G_2$-to-$A_2$ quotient in \S\ref{subsec:new}.

{\tiny\begin{longtable}{rl|rrrrr}
				&
				&$\neg (\vartriangle)$
				&$\neg (\pm)$
				&$\neg (+)$
				&$\varphi_{1212}$-inflation
				&$\varphi_\circ$-inflation\\
				\hline	
$\phi_{1,0}$	&$G_2$
				&&&&$\checkmark$&$\checkmark$\\
$\phi_{1,3}'$	&$G_2(a_1)(21)$
				&&&&$\checkmark$\\
$\phi_{2,1}$	&$G_2(a_1)$
				&&&$\checkmark$\\
$\phi_{2,2}$	&$\tilde A_1$
				&&&$\checkmark$&&$\checkmark$\\
$\phi_{1,3}''$	&$A_1$
				&&&&$\checkmark$\\
$\phi_{1,6}$	&$1$
				&&&&$\checkmark$&$\checkmark$\\
\end{longtable}}

The failure of $(+)_{\psi, G}^\z$ for $\psi = \phi_{2,2}$ is the reason for the complication involving the $G_2$-to-$A_2$ quotient in \Cref{conj:main}.

\subsection{Type $B_3$}

\begin{comment}
The table of $(\chi, \spr_{\psi, G})_W$:
\begin{align}
\begin{array}{r|rrrrrrrrrr}
&7&511^{(11)}&511&331^{(11)}&331&322&31111^{(11)}&31111&22111&\varepsilon\\ 
\hline 
{3.}&1&&1&&1&1&&1&1&1\\
{21.}&&1&&&&1&1&1&1&2\\
{2.1}&&&1&&1&1&&2&2&3\\
{.3}&&&&1&&&&&1&1\\
{1.2}&&&&&1&1&&1&2&3\\
{11.1}&&&&&&1&1&1&1&3\\
{111.}&&&&&&&1&&&1\\
{1.11}&&&&&&&&1&1&3\\
{.21}&&&&&&&&&1&2\\
{.111}&&&&&&&&&&1\\
\end{array}
\end{align}
\end{comment}
%In CHEVIE, $(s_1s_2)^4 = (s_2s_3)^3 = e$.
%(Thus, $s_1$ corresponds to the short simple root.)
The $\alpha_{\psi, G}^\z$:
\begin{align}
\tiny\begin{array}{r|rrrrrrrrrrr}
		&e		&3, 2	&1		&32, 23	&21, 12	&13		&212	&321,213,132,123	&232	&121	&3212, 2123\\
\hline 
3.		&		&		&		&		&		&		&		&(1010)				&		&		&(10202)\\
21.		&		&		&		&		&(101)	&		&(1020)	&(10)				&		&(1020)	&(203)\\
2.1		&		&		&		&		&(101)	&		&(1020)	&					&		&(1020)	&(101)\\
.3		&		&		&		&(101)	&		&		&		&(10)				&(1020)	&		&1\\
1.2		&		&		&		&(101)	&		&		&		&(10)				&(1020)	&		&1\\
11.1	&		&		&		&		&		&(102)	&		&(10)				&		&		&\\
111.	&		&		&(10)	&		&		&		&(10)	&					&		&		&1\\
1.11	&		&		&(10)	&		&		&		&		&					&		&-(10)	&\\
.21		&		&(10)	&		&1		&		&		&		&					&		&(10)	&\\
.111	&1		&		&		&		&		&		&		&					&		&		&\\
\end{array}
\end{align} 
\begin{align} 
\tiny\begin{array}{r|rrrrrrr}
			&2321,1321,1213,1232	&2132		&1212		&32123		&21232,23212	&21321,12132	&13212,12123\\
\hline 
3.			&(10202)				&(10202)	&			&(103040)	&(103040)		&(103030)		&(102020)\\
21.			&(102)					&(102)		&(10202)	&(2050)		&(1030)			&(1030)			&(1020)\\
2.1			&						&			&(10202)	&(1020)		&				&(1020)			&(1020)\\
.3			&(102)					&(102)		&			&(10)		&(10)			&(20)			&(10)\\
1.2			&(102)					&(102)		&			&(10)		&(10)			&(10)			&\\
11.1		&						&(102)		&			&			&				&(10)			&\\
111.		&						&			&			&(10)		&				&				&\\
1.11		&\\
.21			&\\
.111		&\\
\end{array}
\end{align}
\begin{align} 
\tiny\begin{array}{r|rrrrrrr}
			&12321		&232123		&212321,123212		&213212,121232	&132123		&121321		&2123212\\
			\hline 
3.			&(103040)	&(1030506)	&(1040606)			&(1030404)		&(1030404)	&(1030404)	&(10407080)\\
21.			&(1030)		&(10304)	&(204)				&(102)			&(10304)	&(102)		&(1030)\\
2.1			&			&			&(102)				&(102)			&(10304)	&(102)		&(1030)\\
.3			&(1030)		&			&(204)				&(102)			&(102)		&(102)		&(2050)\\
1.2			&(1030)		&			&(102)				&				&			&			&(1020)\\
11.1		&			&			&					&				&			&(102)		&\\
111.		&\\
\vdots		&\\
\end{array}
\end{align}
\begin{align} 
\tiny\begin{array}{r|rrrrrr}
			&2132123,1232123	&1212321,1213212 	&21232123	&12123212	&12132123			&w_\circ\\
			\hline
3.			&(10407080)			&(10305060)			&(103050606)&(103050708)&(104080[11]0[12])	&(1030507080)\\
21.			&(1030)				&(10)				&			&			&(102)				&\\
2.1			&(1030)				&(10)				&(10304)	&			&(102)				&\\
.3			&(1030)				&(1020)				&			&(10202)	&(10304)			&\\
1.2			&					&					&(10304)	&			&					&\\
11.1		&\\
111.		&\\
\vdots		&\\
\end{array}
\end{align}

The unique maximal quotient map to a product of symmetric groups is $\varphi_{1212} \colon W \to W(A_1 \times A_2)$.

{\tiny\begin{longtable}{rl|rrrr}
		&
		&$\neg (\vartriangle)$
		&$\neg (\pm)$
		&$\neg (+)$
		&$\varphi_{1212}$-inflation\\
		\hline	
3.		&$7$
		&&&&$\checkmark$\\
21.		&$511(11)$
		&&&&$\checkmark$\\
2.1		&$511$
		&\\
.3		&$331(11)$
		&&&&$\checkmark$\\
1.2		&$331$
		&\\
11.1	&$322$
		&\\
111.	&$31111(11)$
		&&&&$\checkmark$\\
1.11	&$31111$
		&&&$\checkmark$\\
.21		&$22111$
		&&&&$\checkmark$\\
.111	&$1111111$
		&&&&$\checkmark$
\end{longtable}}

\subsection{Type $C_3$}

%The $s_i$ are labeled the same way as in type $B_3$.
%(Now, $s_1$ corresponds to the \emph{long} simple root.)
The $\alpha_{\psi, G}^\z$:
\begin{align}
\tiny\begin{array}{r|rrrrrrrrrrr}
		&e		&3, 2	&1		&32, 23	&21, 12	&13		&212	&321,213,132,123	&232	&121	&3212, 2123\\
		\hline 
3.		&		&		&		&		&		&		&		&(1010)				&		&		&(10202)\\
.3		&		&		&		&		&		&		&		&					&		&		&\\
2.1		&		&		&		&		&		&		&		&-(10)				&		&		&-(102)\\
1.2		&		&		&		&(101)	&		&		&		&(10)				&(1020)	&		&1\\
21.		&		&		&		&		&(101)	&		&(1020)	&(10)				&		&(1020)	&(203)\\
11.1	&		&		&		&		&		&(102)	&		&(10)				&		&		&\\
.21		&		&(10)	&		&1		&1		&		&		&					&		&(10)	&\\
1.11	&		&(10)	&		&1		&		&		&-(10)	&					&		&		&-1\\
111.	&		&		&(10)	&		&1		&		&(10)	&					&		&		&1\\
.111	&1		&		&		&		&		&		&		&					&		&		&\\
\end{array}
\end{align} 
\begin{align} 
\tiny\begin{array}{r|rrrrrrr}
		&2321,1321,1213,1232	&2132		&1212		&32123		&21232,23212	&21321,12132	&13212,12123\\
		\hline 
3.		&(10202)				&(10202)	&			&(103040)	&(103040)		&(103030)		&(102020)\\
.3		&						&			&			&			&				&(10)			&(10)\\
2.1		&-(102)					&-(102)		&			&-(1030)	&-(1030)		&-(10)			&\\
1.2		&(102)					&(102)		&			&(10)		&(10)			&(10)			&\\
21.		&(102)					&(102)		&(10202)	&(2050)		&(1030)			&(1030)			&(1020)\\
11.1	&						&(102)		&			&			&				&(10)			&\\
.21		&						&			&			&			&				&				&\\
1.11	&						&			&			&-(10)\\
111.	&						&			&			&(10)\\
.111	&\\
\end{array}
\end{align}
\begin{align} 
\tiny\begin{array}{r|rrrrrrr}
		&12321		&232123		&212321,123212		&213212,121232	&132123		&121321		&2123212\\
		\hline 
3.		&(103040)	&(1030506)	&(1040606)			&(1030404)		&(1030404)	&(1030404)	&(10407080)\\
.3		&			&			&(102)				&(102)			&(102)		&(102)		&(1030)\\
2.1		&-(1030)	&-(10304)	&-(102)				&				&			&			&\\
1.2		&(1030)		&			&(102)				&				&			&			&(1020)\\
21.		&(1030)		&(10304)	&(204)				&(102)			&(10304)	&(102)		&(1030)\\
11.1	&			&			&					&				&			&(102)		&\\
.21		&\\
1.11	&\\
111.	&\\
.111	&\\
\end{array}
\end{align}
\begin{align} 
\tiny\begin{array}{r|rrrrrr}
		&2132123,1232123	&1212321,1213212 	&21232123	&12123212	&12132123			&w_\circ\\
		\hline
3.		&(10407080)			&(10305060)			&(103050606)&(103050708)&(104080[11]0[12])	&(1030507080)\\
.3		&(1030)				&(1020)				&(10304)	&(10202)	&(10304)			&\\
2.1		&					&					&(10304)	&			&					&\\
1.2		&					&					&			&			&					&\\
21.		&(1030)				&(10)				&			&			&(102)				&\\
11.1	&\\
.21		&\\
\vdots	&\\
\end{array}
\end{align}

As in type $B_3$, it suffices to consider $\varphi_{1212} \colon W \to W(A_1 \times A_2)$.

{\tiny\begin{longtable}{rl|rrrr}
		&
		&$\neg (\vartriangle)$
		&$\neg (\pm)$
		&$\neg (+)$
		&$\varphi_{1212}$-inflation\\
		\hline	
3.		&$6$
		&&&&$\checkmark$\\
.3		&$42(11)$
		&&&&$\checkmark$\\
2.1		&$42$
		&&&$\checkmark$\\
1.2		&$33$
		&\\
21.		&$411$
		&&&&$\checkmark$\\
11.1	&$222$
		&\\
.21		&$2211(11)$
		&&&&$\checkmark$\\
1.11	&$2211$
		&&&$\checkmark$\\
111.	&$21111$
		&&&&$\checkmark$\\
.111	&$111111$
		&&&&$\checkmark$\\
\end{longtable}}

\subsection{Type $B_4$}

%In CHEVIE, $(s_1s_2)^4 = (s_2s_3)^3 = (s_3s_4)^3 = e$.
The unique maximal quotient map to a product of symmetric groups is $\varphi_{1212} \colon W \to W(A_1 \times A_3)$.

{\tiny\begin{longtable}{rl|rrrr}
		&
		&$\neg (\vartriangle)$
		&$\neg (\pm)$
		&$\neg (+)$
		&$\varphi_{1212}$-inflation\\
		\hline	
4.		&$9$
		&&&&$\checkmark$\\
31.		&$711(11)$
		&&&&$\checkmark$\\
3.1		&$711$
		&&&\\
.4		&$531(11,11)$
		&&&&$\checkmark$\\
22.		&$531(2,11)$
		&&&&$\checkmark$\\
2.2		&$531$
		&&&$\checkmark$\\
1.3		&$441$
		&&&\\
21.1	&$522$
		&&&\\
211.	&$51111(11)$
		&&&&$\checkmark$\\
2.11	&$51111$
		&&&\\
11.2	&$333$
		&&&\\
.31		&$33111(11)$
		&&&&$\checkmark$\\
1.21	&$33111$
		&&&\\
111.1	&$32211(11)$
		&&&\\
11.11	&$32211$
		&&&$\checkmark$\\
.22		&$22221$
		&&&&$\checkmark$\\
1111.	&$3111111(11)$
		&&&&$\checkmark$\\
1.111	&$3111111$
		&&&$\checkmark$\\
.211	&$2211111$
		&&&&$\checkmark$\\
.1111	&$111111111$
		&&&&$\checkmark$\\
\end{longtable}}

\subsection{Type $C_4$}

%The $s_i$ are labeled the same way as in type $B_4$.
As in type $B_4$, it suffices to consider $\varphi_{1212} \colon W \to W(A_1 \times A_3)$.

{\tiny\begin{longtable}{rl|rrrr}
		&
		&$\neg (\vartriangle)$
		&$\neg (\pm)$
		&$\neg (+)$
		&$\varphi_{1212}$-inflation\\
		\hline	
4.		&$8$
		&&&&$\checkmark$\\
.4		&$62(11)$
		&&&&$\checkmark$\\
3.1		&$62$
		&$\checkmark$&$\checkmark$&$\checkmark$\\
1.3		&$44(11)$
		&\\
2.2		&$44$
		&\\
31.		&$611$
		&&&&$\checkmark$\\
22.		&$422(11)$
		&&&&$\checkmark$\\
21.1	&$422$
		&\\
.31		&$4211(11)$
		&&&&$\checkmark$\\
2.11	&$4211$
		&&&$\checkmark$\\
11.2	&$332$
		&\\
1.21	&$3311$
		&\\
.22		&$2222(11)$
		&&&&$\checkmark$\\
11.11	&$2222$
		&&&$\checkmark$\\
211.	&$41111$
		&&&&$\checkmark$\\
111.1	&$22211$
		&\\
.211	&$221111(11)$
		&&&&$\checkmark$\\
1.111	&$221111$
		&&&$\checkmark$\\
1111.	&$2111111$
		&&&&$\checkmark$\\
.1111	&$11111111$
		&&&&$\checkmark$\\
\end{longtable}}

\subsection{Type $D_4$}

\begin{comment} 
The table of values $(\chi, \spr_{\psi, G})_W$:
\begin{align}
\scriptsize
\begin{array}{r|rrrrrrrrrrrrr}
&71&53&5111&44+&44-&3311^{(11)}&3311&3221&311111&2222+&2222-&221111&\varepsilon\\
\hline
{.4}&1&1&1&1&1&&1&1&1&1&1&1&1\\
{1.3}&&1&1&1&1&1&1&1&2&2&2&3&4\\
{.31}&&&1&&&&1&1&2&1&1&2&3\\
{2\cdot -}&&&&1&&&1&1&1&2&1&2&3\\
{2\cdot +}&&&&&1&&1&1&1&1&2&2&3\\
{11.2}&&&&&&1&&&1&1&1&3&6\\
{1.21}&&&&&&&1&1&2&2&2&4&8\\
{.22}&&&&&&&&1&1&1&1&1&2\\
{.211}&&&&&&&&&1&&&1&3\\
{11\cdot -}&&&&&&&&&&1&&1&3\\
{11\cdot +}&&&&&&&&&&&1&1&3\\
{1.111}&&&&&&&&&&&&1&4\\
{.1111}&&&&&&&&&&&&&1\\
\end{array}
\end{align}
We write $s, t, y, u$ for the simple reflections respectively numbered $1, 2, 3, 4$ in CHEVIE, so that $(sy)^3 = (ty)^3 = (uy)^3 = e$.
\end{comment} 
%In CHEVIE, $(s_1s_3)^3 = (s_2s_3)^3 = (s_3s_4)^3 = e$.
The maximal quotient maps to products of symmetic groups are $\varphi_{12}, \varphi_{14}, \varphi_{24} \colon W \to W(A_3)$.
The triality of $D_4$ permutes these maps transitively, so it is enough to work with $\varphi_{12}$.

{\tiny\begin{longtable}{rl|rrrr}
			&
			&$\neg (\vartriangle)$
			&$\neg (\pm)$
			&$\neg (+)$
			&$\varphi_{12}$-inflation\\
			\hline	
$.4$		&$71$
			&&&&$\checkmark$\\
$1.3$		&$53$
			&&&$\checkmark$\\
$.31$		&$5111$
			&&&&$\checkmark$\\
$2\cdot-$	&$44\cdot+$
			&\\
$2\cdot+$	&$44\cdot-$
			&\\
$11.2$		&$3311(11)$
			&\\
$1.21$		&$3311$
			&&&$\checkmark$\\
$.22$		&$3221$
			&&&&$\checkmark$\\
$.211$		&$311111$
			&&&&$\checkmark$\\
$11\cdot-$	&$2222\cdot+$
			&\\
$11\cdot+$	&$2222\cdot-$
			&\\
$1.111$		&$221111$
			&\\
$.1111$		&$11111111$
			&&&&$\checkmark$\\
\end{longtable}}

\subsection{Type $F_4$}

%In CHEVIE, $(s_1s_2)^3 = (s_2s_3)^4 = (s_3s_4)^3 = e$, where $s_3$ and $s_4$ correspond to the short simple roots.
The unique maximal quotient map to a product of symmetric groups is $\varphi_{2323} \colon W \to W(A_2 \times A_2)$.

{\tiny\begin{longtable}{rl|rrrr}
				&
				&$\neg (\vartriangle)$
				&$\neg (\pm)$
				&$\neg (+)$
				&$\varphi_{2323}$-inflation\\
				\hline	
$\phi_{1,0}$	&$F_4$
				&&&&$\checkmark$\\
$\phi_{2,4}'$	&$F_4(a_1)(11)$
				&&&&$\checkmark$\\
$\phi_{4,1}$	&$F_4(a_1)$
				&$\checkmark$&$\checkmark$&$\checkmark$\\
$\phi_{2,4}''$	&$F_4(a_2)(11)$
				&&&&$\checkmark$\\
$\phi_{9,2}$	&$F_4(a_2)$
				&$\checkmark$&$\checkmark$&$\checkmark$\\
$\phi_{8,3}'$	&$C_3$
				&\\
$\phi_{8,3}''$	&$B_3$
				&\\
$\phi_{1,12}'$	&$F_4(a_3)(211)$
				&&&&$\checkmark$\\
$\phi_{6,6}''$	&$F_4(a_3)(22)$
				&\\
$\phi_{9,6}'$	&$F_4(a_3)(31)$
				&&&$\checkmark$\\
$\phi_{12,4}$	&$F_4(a_3)$
				&&$\checkmark$&$\checkmark$\\
$\phi_{4,7}'$	&$C_3(a_1)(11)$
				&&&$\checkmark$\\
$\phi_{16,5}$	&$C_3(a_1)$
				&&&$\checkmark$\\
$\phi_{6,6}'$	&$\tilde A_2+A_1$
				&&&$\checkmark$\\
$\phi_{4,8}$	&$B_2(11)$
				&&&&$\checkmark$\\
$\phi_{9,6}''$	&$B_2$
				&&&$\checkmark$\\
$\phi_{4,7}''$	&$A_2+\tilde A_1$
				&&&$\checkmark$\\
$\phi_{8,9}'$	&$\tilde A_2$
				&\\
$\phi_{1,12}''$	&$A_2(11)$
				&&&&$\checkmark$\\
$\phi_{8,9}''$	&$A_2$
				&\\
$\phi_{9,10}$	&$A_1+\tilde A_1$
				&\\
$\phi_{2,16}'$	&$\tilde A_1(11)$
				&&&&$\checkmark$\\
$\phi_{4,13}$	&$\tilde A_1$
				&&&$\checkmark$\\
$\phi_{2,16}''$	&$A_1$
				&&&&$\checkmark$\\
$\phi_{1,24}$	&$1$
				&&&&$\checkmark$\\
\end{longtable}}

\subsection{Type $B_5$}

%In CHEVIE, $(s_1s_2)^4 = (s_2s_3)^3 = (s_3s_4)^3 = (s_4s_5)^3 = e$.
The unique maximal quotient map to a product of symmetric groups is $\varphi_{1212} \colon W \to W(A_1 \times A_4)$.

{\tiny\begin{longtable}{rl|rrrr}
		&
		&$\neg (\vartriangle)$
		&$\neg (\pm)$
		&$\neg (+)$
		&$\varphi_{1212}$-inflation\\
		\hline	
5.		&$[11]$
		&&&&$\checkmark$\\
41.		&$911(11)$
		&&&&$\checkmark$\\
4.1		&$911$
		&\\
.5		&$731(11,11)$
		&&&&$\checkmark$\\
32.		&$731(2,11)$
		&&&&$\checkmark$\\
3.2		&$731$
		&$\checkmark$&$\checkmark$&$\checkmark$\\
1.4		&$551(11)$
		&\\
2.3		&$551$
		&\\
31.1	&$722$
		&\\
311.	&$71111(11)$
		&&&&$\checkmark$\\
3.11	&$71111$
		&\\
22.1	&$533(11)$
		&\\
21.2	&$533$
		&&&$\checkmark$\\
221.	&$53111(11,11)$
		&&&&$\checkmark$\\
.41		&$53111(11,2)$
		&&&&$\checkmark$\\
2.21	&$53111$
		&&&$\checkmark$\\
11.3	&$443$
		&\\
1.31	&$44111$
		&\\
211.1	&$52211(11)$
		&\\
21.11	&$52211$
		&\\
111.2	&$33311(11)$
		&\\
11.21	&$33311$
		&&&$\checkmark$\\
.32		&$33221(11)$
		&&&&$\checkmark$\\
1.22	&$33221$
		&\\
2111.	&$5111111(11)$
		&&&&$\checkmark$\\
2.111	&$5111111$
		&\\
.311	&$3311111(11)$
		&&&&$\checkmark$\\
1.211	&$3311111$
		&\\
111.11	&$3222$
		&\\
1111.1	&$3221111(11)$
		&\\
11.111	&$3221111$
		&&&$\checkmark$\\
.221	&$2222111$
		&&&&$\checkmark$\\
11111.	&$311111111(11)$
		&&&&$\checkmark$\\
1.1111	&$311111111$
		&&&$\checkmark$\\
.2111	&$221111111$
		&&&&$\checkmark$\\
.11111	&$11111111111$
		&&&&$\checkmark$\\
\end{longtable}}

\subsection{Type $C_5$}

%The $s_i$ are labeled the same way as in type $B_5$.
As in type $B_5$, it suffices to consider $\varphi_{1212} \colon W \to W(A_1 \times A_4)$.

{\tiny\begin{longtable}{rl|rrrr}
		&
		&$\neg (\vartriangle)$
		&$\neg (\pm)$
		&$\neg (+)$
		&$\varphi_{1212}$-inflation\\
		\hline	
5.		&$[10]$
		&&&&$\checkmark$\\
.5		&$82(11)$
		&&&&$\checkmark$\\
4.1		&$82$
		&$\checkmark$&$\checkmark$&$\checkmark$\\
1.4		&$64(11)$
		&\\
3.2		&$64$
		&$\checkmark$&$\checkmark$&$\checkmark$\\
41.		&$811$
		&&&&$\checkmark$\\
2.3		&$55$
		&\\
32.		&$622(11)$
		&&&&$\checkmark$\\
31.1	&$622$
		&\\
.41		&$6211(11)$
		&&&&$\checkmark$\\
3.11	&$6211$
		&$\checkmark$&$\checkmark$&$\checkmark$\\
11.3	&$442(11)$
		&\\
21.2	&$442$
		&&&$\checkmark$\\
1.31	&$4411(11)$
		&\\
2.21	&$4411$
		&\\
22.1	&$433$
		&\\
.32		&$4222(11)$
		&&&&$\checkmark$\\
21.11	&$4222$
		&&&$\checkmark$\\
311.	&$61111$
		&&&&$\checkmark$\\
1.22	&$3322(11)$
		&\\
11.21	&$3322$
		&&&$\checkmark$\\
221.	&$42211(11)$
		&&&&$\checkmark$\\
211.1	&$42211$
		&\\
111.2	&$33211$
		&\\
.311	&$421111(11)$
		&&&&$\checkmark$\\
2.111	&$421111$
		&&&$\checkmark$\\
1.211	&$331111$
		&\\
111.11	&$22222$
		&\\
.221	&$222211(11)$
		&&&&$\checkmark$\\
11.111	&$222211$
		&&&$\checkmark$\\
2111.	&$4111111$
		&&&&$\checkmark$\\
1111.1	&$2221111$
		&\\
.2111	&$22111111(11)$
		&&&&$\checkmark$\\
1.1111	&$22111111$
		&&&$\checkmark$\\
11111.	&$211111111$
		&&&&$\checkmark$\\
.11111	&$1111111111$
		&&&&$\checkmark$\\
\end{longtable}}

\subsection{Type $D_5$}

%In CHEVIE, $(s_1s_3)^3 = (s_2s_3)^3 = (s_3s_4)^3 = (s_4s_5)^3 = e$.
The unique maximal quotient map to a product of symmetric groups is $\varphi_{12} \colon W \to W(A_4)$.

{\tiny\begin{longtable}{rl|rrrr}
		&
		&$\neg (\vartriangle)$
		&$\neg (\pm)$
		&$\neg (+)$
		&$\varphi_{12}$-inflation\\
		\hline	
.5		&$91$
		&&&&$\checkmark$\\
1.4		&$73$
		&$\checkmark$&$\checkmark$&$\checkmark$\\
.41		&$7111$
		&&&&$\checkmark$\\
2.3		&$55$
		&\\
11.3	&$5311(11)$
		&&&$\checkmark$\\
1.31	&$5311$
		&$\checkmark$&$\checkmark$&$\checkmark$\\
2.21	&$4411$
		&\\
.32		&$5221$
		&&&&$\checkmark$\\
1.22	&$3331$
		&\\
.311	&$511111$
		&&&&$\checkmark$\\
11.21	&$3322$
		&\\
111.2	&$331111(11)$
		&\\
1.211	&$331111$
		&&&$\checkmark$\\
.221	&$322111$
		&&&&$\checkmark$\\
11.111	&$222211$
		&\\
.2111	&$31111111$
		&&&&$\checkmark$\\
1.1111	&$22111111$
		&\\
.11111	&$1111111111$
		&&&&$\checkmark$\\
\end{longtable}}

\subsection{Type $B_6$, Rationally Smooth $z$}

The unique maximal quotient map to a product of symmetic groups is $\varphi_{1212} \colon W \to W(A_1 \times A_5)$.

There are $65$ irreducible characters of $W$, indexed by the bipartitions of $6$.
Of these, $22$ are inflated along $\varphi_{1212}$, being those indexed by the bipartitions $(\xi, \eta)$ in which one of $\xi$ or $\eta$ is empty, \emph{cf.}~\Cref{ex:bc}.

Properties $(\vartriangle)_{\psi, G}^\z$ and $(\pm)_{\psi, G}^\z$ always hold when $\z$ is rationally smooth.
The $\psi$ where $(+)_{\psi, G}^\z$ fails for some rationally smooth $\z$ are listed below.

{\tiny\begin{longtable}{rl}
4.2		&$931$\\
1.5		&$751(11,11)$\\
3.3		&$751$\\
3.21	&$73111$\\
221.1	&$53311(11,11)$\\
211.2	&$53311(11,2)$\\
21.21	&$53311$\\
11.31	&$44311$\\
2.22	&$53221$\\
2.211	&$5311111$\\
11.211	&$3331111$\\
111.111	&$3222211$\\
11.1111	&$322111111$\\
1.11111	&$31111111111$\\
\end{longtable}}

\subsection{Type $C_6$, Rationally Smooth $z$}

As in type $B_6$, it suffices to consider $\varphi_{1212} \colon W \to W(A_1 \times A_5)$, and the $22$ irreducible characters inflated along $\varphi_{1212}$ are those indexed by the bipartitions $(\xi, \eta)$ in which one of $\xi$ or $\eta$ is empty.

Properties $(\vartriangle)_{\psi, G}^\z$ and $(\pm)_{\psi, G}^\z$ always hold when $\z$ is rationally smooth.
The $\psi$ where $(+)_{\psi, G}^\z$ fails for some rationally smooth $\z$ are listed below.

{\tiny\begin{longtable}{rl}
5.1		&$[10]2$\\
1.5		&$84(11)$\\
4.2		&$84$\\
4.11	&$8211$\\
31.2	&$642$\\
1.41	&$6411(11)$\\
3.21	&$6411$\\
31.11	&$6222$\\
11.31	&$4422(2,11)$\\
21.21	&$4422$\\
22.11	&$4332$\\
3.111	&$621111$\\
221.1	&$43311$\\
21.111	&$422211$\\
11.211	&$332211$\\
111.111	&$222222$\\
2.1111	&$42111111$\\
11.1111	&$22221111$\\
1.11111	&$2211111111$\\
\end{longtable}}

\subsection{Type $D_6$, Rationally Smooth $z$}

The unique maximal quotient map to a product of symmetic groups is $\varphi_{12} \colon W \to W(A_5)$.

%There are $37$ irreducible characters of $W$, of two kinds.
%Following~\cite{carter}, the first kind is indexed by unordered pairs $\{\xi, \eta\}$, where $(\xi, \eta)$ runs over bipartitions of $6$ in which $\xi \neq \eta$; the second kind are indexed by labels $(\xi, \xi) \cdot \epsilon$, where $(\xi, \xi)$ runs over bipartitions of $6$ with both components identical and $\epsilon \in \{\pm\}$.
%There are $11$ irreducibles inflated along $\varphi_{12}$, all of the first kind:
%They are indexed by the pairs $\{\xi, \eta\}$ in which one of $\xi$ or $\eta$ is empty.

%The $\psi$ where at least one of $(\vartriangle)_{\psi, G}^\z, (\pm)_{\psi, G}^\z, (+)_{\psi, G}^\z$ fails for some rationally smooth $\z$ are listed below.

{\tiny\begin{longtable}{rl|rrrr}
			&
			&$\neg (\vartriangle)$
			&$\neg (\pm)$
			&$\neg (+)$
			&$\varphi_{12}$-inflation\\
			\hline	
$.6$		&$[11]1$
			&&&&$\checkmark$\\
$1.5$		&$93$
			&$\checkmark$&&$\checkmark$\\
$.51$		&$9111$
			&&&&$\checkmark$\\
$2.4$		&$75$
			&$\checkmark$&$\checkmark$&$\checkmark$\\
$11.4$		&$7311(11)$
			&&&$\checkmark$\\
$1.41$		&$7311$
			&$\checkmark$&$\checkmark$&$\checkmark$\\
$3\cdot+$	&$66\cdot+$
			&\\
$3\cdot-$	&$66\cdot-$
			&\\
$21.3$		&$5511(11)$
			&\\
$2.31$		&$5511$
			&\\
$.42$		&$7221$
			&&&&$\checkmark$\\
$.33$		&$5331(11)$
			&&&&$\checkmark$\\
$1.32$		&$5331$
			&\\
$.411$		&$711111$
			&&&&$\checkmark$\\
$2.22$		&$4431$
			&\\
$11.31$		&$5322$
			&&&$\checkmark$\\
$111.3$		&$531111(11)$
			&&&$\checkmark$\\
$1.311$		&$531111$
			&&&$\checkmark$\\
$21\cdot+$	&$4422\cdot+$
			&\\
$21\cdot-$	&$4422\cdot-$
			&\\
$2.211$		&$441111$
			&\\
$.321$		&$522111$
			&&&&$\checkmark$\\
$11.22$		&$3333$
			&\\
$1.221$		&$333111$
			&\\
$111.21$	&$332211(11)$
			&\\
$11.211$	&$332211$
			&&&$\checkmark$\\
$.3111$		&$51111111$
			&&&&$\checkmark$\\
$.222$		&$322221$
			&&&&$\checkmark$\\
$1111.2$	&$33111111(11)$
			&\\
$1.2111$	&$33111111$
			&&&$\checkmark$\\
$.2211$		&$32211111$
			&&&&$\checkmark$\\
$111\cdot+$	&$222222\cdot+$
			&\\
$111\cdot-$	&$222222\cdot-$
			&\\
$11.1111$	&$22221111$
			&\\
$.21111$	&$3111111111$
			&&&&$\checkmark$\\
$1.11111$	&$2211111111$
			&\\
$.111111$	&$111111111111$
			&&&&$\checkmark$\\
\end{longtable}}

\subsection{Type $E_6$, Rationally Smooth $z$}

The unique nontrivial quotient map is the sign character $\varepsilon \colon W \to \{\pm 1\} \simeq W(A_1)$.

{\tiny\begin{longtable}{rl|rrrr}
				&
				&$\neg (\vartriangle)$
				&$\neg (\pm)$
				&$\neg (+)$
				&$\varepsilon$-inflation\\
				\hline
$\phi_{1,0}$	&$E_6$
				&&&&$\checkmark$\\
$\phi_{6,1}$	&$E_6(a_1)$
				&$\checkmark$&$\checkmark$&$\checkmark$\\
$\phi_{20,2}$	&$D_5$
				&\\
$\phi_{15,5}$	&$E_6(a_3)(11)$
				&\\
$\phi_{30,3}$	&$E_6(a_3)$
				&&&$\checkmark$\\
$\phi_{15,4}$	&$A_5$
				&\\
$\phi_{64,4}$	&$D_5(a_1)$
				&&&$\checkmark$\\
$\phi_{60,5}$	&$A_4+A_1$
				&\\
$\phi_{24,6}$	&$D_4$
				&\\
$\phi_{81,6}$	&$A_4$
				&\\
$\phi_{20,10}$	&$D_4(a_1)(111)$
				&&&$\checkmark$\\
$\phi_{90,8}$	&$D_4(a_1)(21)$
				&&&$\checkmark$\\
$\phi_{80,7}$	&$D_4(a_1)$
				&&&$\checkmark$\\
$\phi_{60,8}$	&$A_3+A_1$
				&\\
$\phi_{10,9}$	&$2A_2+A_1$
				&\\
$\phi_{81,10}$	&$A_3$
				&\\
$\phi_{60,11}$	&$A_2+2A_1$
				&\\
$\phi_{24,12}$	&$2A_2$
				&\\
$\phi_{64,13}$	&$A_2+A_1$
				&\\
$\phi_{15,17}$	&$A_2(11)$
				&\\
$\phi_{30,15}$	&$A_2$
				&&&$\checkmark$\\
$\phi_{15,16}$	&$3A_1$
				&\\
$\phi_{20,20}$	&$2A_1$
				&\\
$\phi_{6,25}$	&$A_1$
				&\\
$\phi_{1,36}$	&$1$
				&&&&$\checkmark$\\
\end{longtable}}

\appendix 

\section{The Exotic Fourier Transform}\label[appendix]{sec:dl}

For completeness, this appendix provides the definition of the pairing $\{-, -\}$ from the character theory of finite reductive groups~\cite{lusztig_84}.

Throughout, $\bb{k}$ is the algebraic closure of a finite field $\bb{F}$.
As in the body of the article, $G$ is a connected reductive algebraic group over $\bb{k}$ with Weyl group $W$, such that the characteristic of $\bb{k}$ is a good prime for $G$~\cite[28]{carter}.
Let $F \colon G \to G$ be the Frobenius map corresponding to the split $\bb{F}$-structure on $G$.

For all $w \in W$, the Deligne--Lusztig variety $X_w$ is defined by the pullback square:
\begin{equation}
\begin{tikzcd}
	O_w \arrow{d} &X_w \arrow{l} \arrow{d}\\
	\cal{B} \times \cal{B} &\cal{B} \arrow{l}[above]{\id \times F}
\end{tikzcd}
\end{equation}
In other words, $X_w \subseteq \cal{B}$ is the subvariety of Borels $B$ such that $B \xrightarrow{w} F(B)$.
The finite reductive group $G^F$ acts on $X_w$, and hence on its compactly-supported $\ell$-adic \'etale cohomology $\ur{H}_c^i(X_w, \QL)$.
The \dfemph{Deligne--Lusztig character} $R_w \colon G^F \to \QL$ is the virtual character defined by
\begin{align}
	R_w(g) = \sum_i {(-1)^i} \tr(g \mid \ur{H}_c^i(X_w, \QL)).
\end{align}
For any $\psi \in \Irr(W)$, the \dfemph{almost-character} $R_\psi \colon G^F \to \QL$ is defined by
\begin{align} 
	R_\psi = \frac{1}{|W|} \sum_{w \in W} \psi(w)R_w.
\end{align} 
As in \Cref{sec:intro}, let $e$ be the identity of $W$.
Then $X_e$ is the finite set $\cal{B}^F$, and $R_e$ is the character of the principal series representation of $G^F$ on $\ur{H}^0(\cal{B}^F, \QL)$.

Let $q = |\bb{F}|$.
Fix a square root $q^{1/2} \in \QL^\times$.
By a classical theorem of Iwahori, the specialized Hecke algebra $H_W(q) \vcentcolon= H_W(\sf{v})|_{\sf{v} \to q^{1/2}}$ acts on $\ur{H}^0(\cal{B}^F, \QL)$ by operators commuting with $G^F$.
Namely, $\delta_w$ sends the indicator function at a Borel $B$ to the indicator on the set of Borels $B'$ such that $B \xrightarrow{w} B'$.
Therefore, $R_e$ can be promoted to a bitrace $R_e \colon G^F \times H_W(q) \to \QL$.
We have a decomposition
\begin{align}\label{eq:principal-series}
R_e
	= \sum_{\chi \in \Irr(W)}
	\rho_\chi \otimes \chi_\sf{v}|_{\sf{v} \to q^{1/2}},
\end{align}
where $\rho_\chi$ is an irreducible character of $G^F$ called the \dfemph{unipotent principal series character} associated with $\chi$.

Let $(-, -)_{G^F}$ be the scalar product on class functions on $G^F$.
We define the \dfemph{(truncated) exotic Fourier transform} $\{-, -\}$ by setting 
\begin{align} 
\{\chi, \psi\} = (\rho_\chi, R_\psi)_{G^F}
	\quad\text{for all $\chi, \psi$}.
\end{align}
It turns out that this, indeed, only depends on $W$, and not on $\bb{F}$ or $G$.

%An irreducible character of $G^F$ is called \dfemph{unipotent} when it has nonzero scalar product with $R_w$ for some $w\in W$.
%Let $\Uch(G^F) \subseteq \Irr(G^F)$ be the set of unipotent irreducible characters.
%Let $\sim$ be the reflexive, symmetric relation on $\Uch(G^F)$ where $\rho \sim \rho'$ if and only if $\rho$ and $\rho'$ both have nonzero scalar product with $R_\psi$ for some $\psi \in \Irr(W)$.
%Lusztig's \dfemph{families} of unipotent irreducible characters are the equivalence classes of $\Uch(G^F)$ under the transitive closure of $\sim$~\cite[\S{12.3}]{carter}.

%The map $\chi \mapsto \rho_\chi$ gives us an embedding of $\Irr(W)$ into $\Uch(G^F)$.
%The \dfemph{(truncated) families} of $\Irr(W)$ are the restrictions of the families above to $\Irr(W)$.
%For an alternative definition based on the two-sided cells of $W$, see~\cite[\S{12.4}]{carter}.

%------------------------------------------------------------

\bibliographystyle{alphaurl}
\bibliography{haiman}

@article{an_25,
	author = {Abreu, Alex and Nigro, Antonio},
	title = {A geometric approach to characters of {Hecke} algebras},
	fjournal = {Journal f{\"u}r die Reine und Angewandte Mathematik},
	journal = {J. Reine Angew. Math.},
	issn = {0075-4102},
	volume = {821},
	pages = {53--114},
	year = {2025},
	doi = {10.1515/crelle-2024-0098},
}

@incollection {achar_rims,
	AUTHOR = {Achar, Pramod N.},
	TITLE = {Springer theory for complex reflection groups},
	BOOKTITLE = {RIMS K{\^{o}}ky{\^{u}}roku 1647. Expansion of Combinatorial Representation Theory. Ed. Hyohe Miyachi},
	PUBLISHER = {Research Institute for Mathematical Sciences},
	YEAR = {2009},
	PAGES = {97--112},
	URL = {https://www.kurims.kyoto-u.ac.jp/~kyodo/kokyuroku/contents/1647.html},
}

@book {achar,
	AUTHOR = {Achar, Pramod N.},
	TITLE = {Perverse sheaves and applications to representation theory},
	SERIES = {Mathematical Surveys and Monographs},
	VOLUME = {258},
	PUBLISHER = {American Mathematical Society, Providence, RI},
	YEAR = {[2021] \copyright 2021},
	PAGES = {xii+562},
	ISBN = {978-1-4704-5597-2},
	DOI = {10.1090/surv/258},
}

@article{aa,
	author = {Achar, P. N. and Aubert, A.-M.},
	title = {Springer correspondences for dihedral groups.},
	fjournal = {Transformation Groups},
	journal = {Transform. Groups},
	issn = {1083-4362},
	volume = {13},
	number = {1},
	pages = {1--24},
	year = {2008},
	doi = {10.1007/s00031-008-9004-2},
}

@article{ap,
	author = {Alexandersson, Per and Panova, Greta},
	title = {{LLT} polynomials, chromatic quasisymmetric functions and graphs with cycles},
	fjournal = {Discrete Mathematics},
	journal = {Discrete Math.},
	issn = {0012-365X},
	volume = {341},
	number = {12},
	pages = {3453--3482},
	year = {2018},
	doi = {10.1016/j.disc.2018.09.001},
}

@book{agvdd,
	editor = {Artin, M. and Grothendieck, A. and Verdier, J. L. and Deligne, P. and Saint-Donat, Bernard},
	title = {S{\'e}minaire de g{\'e}om{\'e}trie alg{\'e}brique du {Bois}-{Marie} 1963--1964. {Th{\'e}orie} des topos et cohomologie {\'e}tale des sch{\'e}mas ({SGA} 4). {Un} s{\'e}minaire dirig{\'e} par {M}. {Artin}, {A}. {Grothendieck}, {J}. {L}. {Verdier}. {Avec} la collaboration de {P}. {Deligne}, {B}. {Saint}-{Donat}. {Tome} 3. {Expos{\'e}s} {IX} {\`a} {XIX}},
	fseries = {Lecture Notes in Mathematics},
	series = {Lect. Notes Math.},
	issn = {0075-8434},
	volume = {305},
	year = {1973},
	publisher = {Springer, Cham},
	doi = {10.1007/BFb0070714},
}

@incollection{aubert,
	author = {Aubert, Anne-Marie},
	title = {Some properties of character sheaves},
	booktitle = {Olga Taussky-Todd: In memoriam.},
	isbn = {1-57146-051-9},
	pages = {pac. j. math., spec. issue, 37--52},
	year = {1998},
	publisher = {Cambridge, MA: International Press},
}

@article{bc,
	author = {Brosnan, Patrick and Chow, Timothy Y.},
	title = {Unit interval orders and the dot action on the cohomology of regular semisimple {Hessenberg} varieties},
	fjournal = {Advances in Mathematics},
	journal = {Adv. Math.},
	issn = {0001-8708},
	volume = {329},
	pages = {955--1001},
	year = {2018},
	doi = {10.1016/j.aim.2018.02.020},
}

@misc{bhl,
	author = {Brosnan, Patrick and Hong, Jaehyun and Lee, Donggun},
	title = {Geometry of regular semisimple {Lusztig} varieties},
	year = {2025},
	eprint = {2504.15868},
	eprinttype = {arxiv},
	note = {v1},
}

@article{cm,
	author = {Carlsson, Erik and Mellit, Anton},
	title = {A proof of the shuffle conjecture},
	fjournal = {Journal of the American Mathematical Society},
	journal = {J. Am. Math. Soc.},
	issn = {0894-0347},
	volume = {31},
	number = {3},
	pages = {661--697},
	year = {2018},
	doi = {10.1090/jams/893},
}

@book{carter,
	AUTHOR = {Carter, Roger W.},
	TITLE = {Finite groups of {L}ie type},
	SERIES = {Wiley Classics Library},
	NOTE = {Conjugacy classes and complex characters,
	Reprint of the 1985 original,
	A Wiley-Interscience Publication},
	PUBLISHER = {John Wiley \& Sons, Ltd., Chichester},
	YEAR = {1993},
	PAGES = {xii+544},
}

@article{chss,
	author = {Clearman, Samuel and Hyatt, Matthew and Shelton, Brittany and Skandera, Mark},
	title = {Evaluations of {Hecke} algebra traces at {Kazhdan}-{Lusztig} basis elements},
	fjournal = {The Electronic Journal of Combinatorics},
	journal = {Electron. J. Comb.},
	issn = {1077-8926},
	volume = {23},
	number = {2},
	pages = {research paper p2.7, 56},
	year = {2016},
	url = {www.combinatorics.org/ojs/index.php/eljc/article/view/v23i2p7},
}

@misc{deligne,
	author = {Deligne, Pierre},
	title = {Th{\'e}oremes de finitude en cohomologie {$\ell$}-adique (avec un appendice par {L}. {Illusie})},
	year = {1977},
	language = {French},
	howpublished = {Semin. {Geom}. algebr. {Bois}-{Marie}, {SGA} 4 1/2, {Lect}. {Notes} {Math}. 569, 233-261},
	doi = {10.1007/bfb0091524},
}

@article{ghmp,
	author = {Geck, Meinolf and Hiss, Gerhard and L{\"u}beck, Franz and Malle, Gunter and Pfeiffer, G{\"o}tz},
	title = {\textsf{CHEVIE} -- {A} system for computing and processing generic character tables for finite groups of {L}ie type, {W}eyl groups and {H}ecke algebras},
	journal = {Appl. Algebra Engrg. Comm. Comput.},
	volume = {7},
	pages = {175--210},
	year = {1996},
}

@article{gm,
	author = {Geck, Meinolf and Malle, Gunter},
	title = {On special pieces in the unipotent variety},
	fjournal = {Experimental Mathematics},
	journal = {Exp. Math.},
	issn = {1058-6458},
	volume = {8},
	number = {3},
	pages = {281--290},
	year = {1999},
	doi = {10.1080/10586458.1999.10504405},
}

@misc{gmrww,
	author = {Griffin, Sean T. and Mellit, Anton and Romero, Marino and Weigl, Kevin and Wen, Joshua Jeishing},
	title = {On {Macdonald} expansions of {$q$}-chromatic symmetric functions and the {Stanley}-{Stembridge} {Conjecture}},
	year = {2025},
	eprint = {2504.06936},
	eprinttype = {arXiv},
	note = {v1},
}

@article{haiman,
	author = {Haiman, Mark},
	title = {Hecke algebra characters and immanant conjectures},
	fjournal = {Journal of the American Mathematical Society},
	journal = {J. Am. Math. Soc.},
	issn = {0894-0347},
	volume = {6},
	number = {3},
	pages = {569--595},
	year = {1993},
	doi = {10.2307/2152777},
}

@article{he,
	author = {He, Xuhua},
	title = {On affine {Lusztig} varieties},
	fjournal = {Annales Scientifiques de l'{\'E}cole Normale Sup{\'e}rieure. Quatri{\`e}me S{\'e}rie},
	journal = {Ann. Sci. {\'E}c. Norm. Sup{\'e}r. (4)},
	issn = {0012-9593},
	volume = {58},
	number = {3},
	pages = {749--776},
	year = {2025},
	language = {English},
	doi = {10.24033/asens.2615},
	zbMATH = {8096909}
}

@misc{hikita,
	author = {Hikita, Tatsuyuki},
	title = {A proof of the {Stanley}--{Stembridge} conjecture},
	year = {2025},
	eprint = {2410.12758},
	eprinttype = {arXiv},
	note = {v2},
}

@misc{kato,
	author = {Kato, Syu},
	title = {A geometric realization of the chromatic symmetric function of a unit interval graph},
	year = {2024},
	eprint = {2410.12231},
	eprinttype = {arxiv},
	note = {v2},
}

@article{kl,
	author = {Kazhdan, David and Lusztig, George},
	title = {Representations of {Coxeter} groups and {Hecke} algebras},
	fjournal = {Inventiones Mathematicae},
	journal = {Invent. Math.},
	issn = {0020-9910},
	volume = {53},
	pages = {165--184},
	year = {1979},
	doi = {10.1007/BF01390031},
}

@article{kim,
	author = {Kim, Dongkwan},
	title = {On total {Springer} representations for classical types},
	fjournal = {Selecta Mathematica. New Series},
	journal = {Sel. Math., New Ser.},
	issn = {1022-1824},
	volume = {24},
	number = {5},
	pages = {4141--4196},
	year = {2018},
	doi = {10.1007/s00029-018-0438-7},
}

@article{kv,
	author = {Kottwitz, Robert and Viehmann, Eva},
	title = {Generalized affine {Springer} fibres},
	fjournal = {Journal of the Institute of Mathematics of Jussieu},
	journal = {J. Inst. Math. Jussieu},
	issn = {1474-7480},
	volume = {11},
	number = {3},
	pages = {569--609},
	year = {2012},
	doi = {10.1017/S147474801100020X},
}

@incollection{lusztig_79,
	author = {Lusztig, G.},
	title = {On the reflection representation of a finite {Chevalley} group},
	booktitle = {Representation theory of {Lie} groups, {Proc}. {SRC}/{LMS} {Res}. {Symp}., {Oxford} 1977},
	year = {1979},
	volume = {34},
	series = {{Lond}. {Math}. {Soc}. {Lect}. {Note} {Ser}.},
	pages = {325--337},
}

@book {lusztig_84,
	AUTHOR = {Lusztig, George},
	TITLE = {Characters of reductive groups over a finite field},
	SERIES = {Annals of Mathematics Studies},
	VOLUME = {107},
	PUBLISHER = {Princeton University Press, Princeton, NJ},
	YEAR = {1984},
	PAGES = {xxi+384},
	ISBN = {0-691-08350-9; 0-691-08351-7},
	MRCLASS = {20G05 (14L20 20C15)},
	MRNUMBER = {742472},
	MRREVIEWER = {Bhama\ Srinivasan},
	DOI = {10.1515/9781400881772},
}

@article{lusztig_85_2,
	author = {Lusztig, George},
	title = {Character sheaves. {II}},
	fjournal = {Advances in Mathematics},
	journal = {Adv. Math.},
	issn = {0001-8708},
	volume = {57},
	pages = {226--265},
	year = {1985},
	doi = {10.1016/0001-8708(85)90064-7},
}

@article{lusztig_85_3,
	author = {Lusztig, George},
	title = {Character sheaves. {III}},
	fjournal = {Advances in Mathematics},
	journal = {Adv. Math.},
	issn = {0001-8708},
	volume = {57},
	pages = {266--315},
	year = {1985},
	doi = {10.1016/0001-8708(85)90065-9},
}

@article{lusztig_86_5,
	author = {Lusztig, George},
	title = {Character sheaves. {V}},
	fjournal = {Advances in Mathematics},
	journal = {Adv. Math.},
	issn = {0001-8708},
	volume = {61},
	pages = {103--155},
	year = {1986},
	doi = {10.1016/0001-8708(86)90071-X},
}

@article{lusztig_88,
	author = {Lusztig, George},
	title = {Cuspidal local systems and graded {Hecke} algebras. {I}},
	fjournal = {Publications Math{\'e}matiques},
	journal = {Publ. Math., Inst. Hautes {\'E}tud. Sci.},
	issn = {0073-8301},
	volume = {67},
	pages = {145--202},
	year = {1988},
	doi = {10.1007/BF02699129},
}

@incollection {lusztig_93,
	AUTHOR = {Lusztig, George},
	TITLE = {Appendix: {C}oxeter groups and unipotent representations},
	NOTE = {Repr\'esentations unipotentes g\'en\'eriques et blocs des groupes r\'eductifs finis},
	JOURNAL = {Ast\'erisque},
	FJOURNAL = {Ast\'erisque},
	NUMBER = {212},
	YEAR = {1993},
	PAGES = {191--203},
	ISSN = {0303-1179,2492-5926},
	URL = {http://www.numdam.org/item/AST_1993__212__191_0/},
}

@article {lusztig_94,
	AUTHOR = {Lusztig, George},
	TITLE = {Exotic {F}ourier transform},
	NOTE = {With an appendix by Gunter Malle},
	JOURNAL = {Duke Math. J.},
	FJOURNAL = {Duke Mathematical Journal},
	VOLUME = {73},
	YEAR = {1994},
	NUMBER = {1},
	PAGES = {227--241, 243--248},
	ISSN = {0012-7094,1547-7398},
	DOI = {10.1215/S0012-7094-94-07309-2},
}

@article{lusztig_97,
	author = {Lusztig, G.},
	title = {Notes on unipotent classes},
	fjournal = {The Asian Journal of Mathematics},
	journal = {Asian J. Math.},
	issn = {1093-6106},
	volume = {1},
	number = {1},
	pages = {194--207},
	year = {1997},
	doi = {10.4310/AJM.1997.v1.n1.a7},
}

@article{lusztig_11_partitions,
	author = {Lusztig, George},
	title = {On some partitions of a flag manifold.},
	fjournal = {The Asian Journal of Mathematics},
	journal = {Asian J. Math.},
	issn = {1093-6106},
	volume = {15},
	number = {1},
	pages = {1--8},
	year = {2011},
	doi = {10.4310/AJM.2011.v15.n1.a1},
}

@book{macdonald,
	author = {Macdonald, Ian Grant},
	title = {Symmetric functions and {Hall} polynomials.},
	edition = {2nd},
	isbn = {0-19-853489-2},
	year = {1995},
	publisher = {Oxford: Clarendon Press},
}

@article{maxwell,
	author = {Maxwell, George},
	title = {The normal subgroups of finite and affine {Coxeter} groups},
	fjournal = {Proceedings of the London Mathematical Society. Third Series},
	journal = {Proc. Lond. Math. Soc. (3)},
	issn = {0024-6115},
	volume = {76},
	number = {2},
	pages = {359--382},
	year = {1998},
	doi = {10.1112/S0024611598000112},
}

@article{michel,
	author = {Michel, Jean},
	title = {Tower equivalence and {Lusztig}'s truncated {Fourier} transform},
	fjournal = {Proceedings of the American Mathematical Society. Series B},
	journal = {Proc. Am. Math. Soc., Ser. B},
	issn = {2330-1511},
	volume = {10},
	pages = {252--261},
	year = {2023},
	doi = {10.1090/bproc/167},
}

@incollection{rr,
	author = {Rider, Laura and Russell, Amber},
	title = {Perverse sheaves on the nilpotent cone and {Lusztig}'s generalized {Springer} correspondence},
	booktitle = {Lie algebras, Lie superalgebras, vertex algebras and related topics. 2012--2014 Southeastern Lie theory workshop series: Categorification of quantum groups and representation theory, North Carolina State University, Raleigh, NC, USA, April 21--22, 2012},
	series = {Proc. Sympos. Pure Math.},
	volume = {92},
	isbn = {978-1-4704-1844-1; 978-1-4704-3013-9},
	pages = {273--292},
	year = {2016},
	publisher = {Providence, RI: American Mathematical Society (AMS)},
}

@article{sw,
	author = {Shareshian, John and Wachs, Michelle L.},
	title = {Chromatic quasisymmetric functions},
	fjournal = {Advances in Mathematics},
	journal = {Adv. Math.},
	issn = {0001-8708},
	volume = {295},
	pages = {497--551},
	year = {2016},
	doi = {10.1016/j.aim.2015.12.018},
}

@article{sommers,
	author = {Sommers, Eric},
	title = {A generalization of the {Bala}-{Carter} theorem for nilpotent orbits},
	fjournal = {IMRN. International Mathematics Research Notices},
	journal = {Int. Math. Res. Not.},
	issn = {1073-7928},
	volume = {1998},
	number = {11},
	pages = {539--562},
	year = {1998},
	doi = {10.1155/S107379289800035X},
}

@article{springer_78,
	author = {Springer, T. A.},
	title = {A construction of representations of {Weyl} groups},
	fjournal = {Inventiones Mathematicae},
	journal = {Invent. Math.},
	issn = {0020-9910},
	volume = {44},
	pages = {279--293},
	year = {1978},
	doi = {10.1007/BF01403165},
}

@article{stanley,
	author = {Stanley, Richard P.},
	title = {A symmetric function generalization of the chromatic polynomial of a graph},
	fjournal = {Advances in Mathematics},
	journal = {Adv. Math.},
	issn = {0001-8708},
	volume = {111},
	number = {1},
	pages = {166--194},
	year = {1995},
	language = {English},
	doi = {10.1006/aima.1995.1020},
	zbMATH = {755384},
	Zbl = {0831.05027}
}

@article{ss,
	author = {Stanley, Richard P. and Stembridge, John R.},
	title = {On immanants of {Jacobi}-{Trudi} matrices and permutations with restricted position},
	fjournal = {Journal of Combinatorial Theory. Series A},
	journal = {J. Comb. Theory, Ser. A},
	issn = {0097-3165},
	volume = {62},
	number = {2},
	pages = {261--279},
	year = {1993},
	doi = {10.1016/0097-3165(93)90048-D},
}

@article{treumann,
	author = {Treumann, David},
	title = {A topological approach to induction theorems in {Springer} theory.},
	fjournal = {Representation Theory},
	journal = {Represent. Theory},
	issn = {1088-4165},
	volume = {13},
	pages = {8--18},
	year = {2009},
	doi = {10.1090/S1088-4165-09-00342-2},
}

@misc {trinh,
	author = {Trinh, Minh-T\^am Quang},
	title = {From the {H}ecke Category to the Unipotent Locus}, 
	year = {2021},
	eprint = {2106.07444},
	eprinttype = {arXiv},
	note = {v2},
}

@book{wang,
	author = {Wang, Xiao Griffin},
	title = {Multiplicative {H}itchin Fibrations and the {F}undamental {L}emma},
	year = {2026},
	publisher = {Princeton University Press},
	address = {Princeton},
	note = {{ISBN:} 9780691281254},
}

@article{ww,
	author = {Webster, Ben and Williamson, Geordie},
	title = {The geometry of {Markov} traces},
	fjournal = {Duke Mathematical Journal},
	journal = {Duke Math. J.},
	issn = {0012-7094},
	volume = {160},
	number = {2},
	pages = {401--419},
	year = {2011},
	doi = {10.1215/00127094-1444268},
}

\end{document}